\documentclass[]{article}

\usepackage{amsmath}
\usepackage{amssymb} 
\usepackage{amsthm}
\usepackage{array}
\usepackage{graphicx}
\usepackage[square,numbers]{natbib}
\newtheorem{theorem}{Theorem}[section]

\usepackage{datetime}

\title{Distribution of twin primes in repeating sequences of prime factors}

\author{John K Sellers}

\begin{document}

\maketitle

\begin{abstract}
	We present a deterministic relationship between relative primes and twin primes in successively larger sequences of the natural numbers. This enables setting a finite lower limit on the occurrence of actual twin primes in an unbounded selection of domains.	
\end{abstract}

\section{Natural Numbers as Repeating  Sequences of Prime Factors}

The formulation presented here differs from other approaches to prime pairs as outlined in \cite{M} and which are covered in detail in \cite{BFI1,BFI2,GPY1,GPY2,Z}. Here we determine the locations of relative primes, prime to $P\le P_{k}$, in successively larger sequences of the natural numbers.  
Those are the only locations where actual prime numbers $P>P_{k}$ can appear. Therefore we refer to numbers in those locations as prospective prime numbers.

This approach is more direct and less abstract where the key is tracking prospective prime numbers not prime numbers themselves.
In defining a finite set covering all combinations of prime factors $P_{1}$, $P_{2}$, $P_{3}$, $\cdots,P_{k}$ relative to each other, we can count the number of prospective prime numbers and prospective twin primes in the set, specify their locations, and determine the uniformity of their distribution. 

The tie to actual twin primes occurs when a prospective twin prime, prime to $P\le P_{k}$, occurs between $P_{k}$ and $P_{k+1}^2$. Then they are actual twin primes. Given this, we can place a non-vacant lower bound on the number of such twin primes for an unbounded succession of specific choices of $P_{k}$.

\subsection{Numbers in Relation to Prime Factors $2$ and $3$}

If one considers a sequence of natural numbers greater than $5$ in terms of its relationships to prime factors $2$ and $3$ only, one sees a repeating pattern of six numbers relative to those factors as shown in Figure~\ref{T: elementalsequence}. The bottom row shows a useful representation, which we refer to as "sequence notation," that explicitly shows the count, $n$, of the repetitions of the $6$-number sequence and a count, "$i$", within each repeated sequence. The count within the repeating sequence shows an arbitrary but convenient choice for its start that has the two numbers prime to $2$ and $3$ as the first and third numbers in the repeating sequence.

\begin{figure}
	\begin{center}
		\renewcommand{\arraystretch}{1.5}
		\begin{tabular}{|c|c|c|c|c|c|}
			\hline 
			N & N+1 & N+2 & N+3 & N+4 & N+5 \\ 
			\hline 
			$\widetilde{P}$ & $2|$, $3|$ & $\widetilde{P}$ & $2|$ & $3|$ & $2|$ \\ 
			\hline 
			$N_{n,1}$ & $N_{n,2}$ & $N_{n,3}$ & $N_{n,4}$ & $N_{n,5}$ &$ N_{n,6}$ \\ 
			\hline 
		\end{tabular}
		\caption{ Underlying Sequence of Natural Numbers Relative to Prime Factors $2$ and $3$, The figure as an example shows an arbitrary starting point of $N=6n+5$. $\widetilde{P}$ represents a prospective prime number because it marks the only location where a prime number $>3$ may occur.} 
		\label{T: elementalsequence}
	\end{center}
\end{figure}

Each sequence repeats with the identical pattern of prime factors $2$ and $3$ because: $N_{n,i}=N_{n',i}+6(n-n'$) shows that both  $N_{n,i}$ and $N_{n',i}$ share the same combination of the prime factors $2$ and $3$. Therefore the two locations occupied by numbers prime to $2$ and $3$ in each $6$-number sequence represent the only locations where prime numbers can occur and each occurrence has  the potential to be a twin prime. Whether they are  prime numbers or twin primes of course depends on where they occur in the overall sequence of natural numbers in relation to all other prime factors. Since we retain the notation citing repetitions of the $6$-number sequence and the position within a given repetition of that sequence even as we consider larger sequences with other prime factors, we refer to the $6$-number sequence as the elemental sequence, where the first ($zero^{th}$), such elemental sequence starts at the number $5$, i.e.,

\[ 
N_{0,1}=5,N_{0,2}=6,N_{0,3}=7,N_{0,4}=8,N_{0,5}=9,N_{0,6}=10
\]

An immediate observation from this formulation is that any prime number, $P>3$, may be represented as either $5+6n$, Progression 1, or $7+6n$, Progression 2, where $(n \ge 0)$ or equivalently as: $6n' \pm 1,  (n' \ge 1)$. 

\subsection{Conversion in and out of sequence notation}\label{S: finding_$N_{n.i}$}
As just introduced we use the representation of natural numbers,  $N_{n, i}$, starting at  $N_{0, 1}=5$, where $i$ represents the number's position in the elemental sequence and $n$ represents which elemental sequence.  For example, $5,6,7,8,9,10$ represents the ${zero}^{th }$ elemental sequence. Then other natural numbers are represented by 

\begin{equation}\label{E: Sequence_of_natural_numbers_1}
	N=N_{n, i}=5+6n+(i-1) \quad \textrm{with} \quad i=1,2,3,4,5,6
\end{equation}

Given any natural number, $N$, one can find its representation in terms of its elemental sequence and part of that sequence by choosing $t= 0,1,2,3,4,5$ so that $6 |  (N-5-t)$, giving:

\begin{equation}\label{E: Sequence_of_natural_numbers_2}
	n= \frac{ (N-5-t)}{6} \quad \textrm{and} \quad i=t+1
\end{equation}

Testing for values of $t$ is easy, but one can also solve for $n$ and $i$ in terms of $N$ as follows:\\

Note that $N-5=\left\lfloor \frac{N-5}{6} \right\rfloor 6 +(N-5)\bmod{6}$. Then comparing this with Equation~(\ref{E: Sequence_of_natural_numbers_1}) gives:

\begin{equation}
	n=\left\lfloor \frac{N-5}{6} \right\rfloor\quad \textrm{and} \quad i=(N-5)\bmod{6}+1
\end{equation}\\

in turn giving:

\begin{equation}
	N=N_{ \left\lfloor\frac{N-5}{6} \right\rfloor,(N-5)\bmod{6}+1}	
\end{equation}

Numbers in the first or third position of the elemental sequence may be represented as the following arithmetic progressions:

\[
N_{n,1}=6n+5 \quad \textrm{and}\quad N_{n,3}=6n+7
\]

\subsection{Products of Numbers Represented in the Sequence Notation}\label{S: ProductsofNumbers}

Multiplication rules for Numbers in the first and third positions of the elemental sequence:

\begin{align}\label{E: sequenceproducts1}\notag
	N_{n_{i},1}N_{n_{j},1}&=N_{[6n_{i}n_{j}+5(n_{i}+n_{j})+3],3}     \\ 
	N_{n_{i},1}N_{n_{j},3}&=N_{[6n_{i}n_{j}+7n_{i}+5n_{j}+5],1}     \\\notag
	N_{n_{i},3}N_{n_{j},3}&=N_{[ 6n_{i}n_{j}+7(n_{i}+n_{j})+7],3}    \\ \notag
\end{align}

Define $P_{n,1}=N_{n,1}$ and $P_{n,3}=N_{n,3}$ when $N_{n,1}$ and $N_{n,3}$  respectively are prime numbers. Then In terms of products of prime numbers the multiplication equations become:
\begin{align}\label{E: sequenceproducts2}\notag
	P_{n_{i},1}P_{n_{j},1}&=N_{[6n_{i}n_{j}+5(n_{i}+n_{j})+3],3}     \\ 
	P_{n_{i},1}P_{n_{j},3}&=N_{[6n_{i}n_{j}+7n_{i}+5n_{j}+5],1}     \\\notag
	P_{n_{i},3}P_{n_{j},3}&=N_{[ 6n_{i}n_{j}+7(n_{i}+n_{j})+7],3}    \\ \notag
\end{align}

Summarizing in words: 

1) Any product of prime numbers $>3$ that involves an odd number of prime numbers from Progression 1 (all of the form $P_{n,1}$) produces a number in Progression 1: $N_{n',1}$.

2) Any other product of prime numbers $>3$ produces a number in Progression 2: $N_{n',3}$.

\section{Prospective Prime Numbers and Prospective twin primes in Successively Larger Sequences of The Natural Numbers}\label{S: TPC}

\subsection{Definitions and  Properties}\label{S: defs}

In this section we add some formalism to a discussion on the sequences of prime factors  underlying the natural numbers. Our intention is to add some precision and clarity to this view as we consider successively larger sequences of natural numbers and their implications for the occurrence of prime numbers and twin primes.

It is key to this perspective of the natural numbers to note that the number $N$ and $N+n\prod_{i=1}^{k}P_{i}$, where $n$ is an integer, $0 \le n \le P_{k+1}-1$, have the same occurrence of distinct prime factors among $P \le P_{k}$. Therefore the sequences 
\[
N, N+1, N+2, \cdots
\]
and
\[
N+n\prod_{i=1}^{k}P_{i} \; , \; N+n\prod_{i=1}^{k}P_{i}+1 \; , \; N+n\prod_{i=1}^{k}P_{i}+2, \cdots
\]
have the same sequence of prime factors $P \le P_{k}$ and that sequence repeats at a frequency of every $\prod_{i=1}^{k+1}P_{i}$ numbers.

Define $S_{k}$ as the following ordered sequence of natural numbers:

\begin{equation}\label{E: def$S_{k}$}
	S_{k} = \left\{N_{i}:   5 \le N_{i} \le 4 + \prod_{i=1}^{k}P_{i} \quad \& \quad N_{i+1}=N_{i}+1 \right\}
\end{equation}  

Note that sets $S_{k}$ each contain $\prod_{i=1}^{k}P_{i}$ successive natural numbers, all starting with the number $5$, but ending on successively larger numbers as we increment $k$.

Define $\psi_{k}$ as the underlying ordered sequence of prime factors $P \le P_{k}$ found in $S_{k}$. For example, given that $S_{2}=\{5,6,7,8,9,10  \}$, then $\psi_{2}=\{\widetilde{P}, 2 \cdot 3, \widetilde{P}, 2,3,2  \}$, where $\widetilde{P}$ indicates a number prime relative to all $P \le P_{2}$ (i.e., $P \le 3$). The values in $\psi_{2}$ are the distinct prime factors $P \le P_{2}$ found in the numbers of $S_{2}$ that are in the same sequencial position as the corresponding values in $\psi_{2}$. 

When we extend $S_{k}$ to $S_{k+1}$, the underlying sequence of prime factors  $\psi_{k}$ of $S_{k}$ repeats $P_{k+1}$ times in $S_{k+1}$, giving:

\begin{equation}\label{E: repeat_psi}
	\widetilde{\psi}_{k+1} =\{\psi_{k}^{(0)},\psi_{k}^{(1)}, \psi_{k}^{(2)}, \cdots, \psi_{k}^{(P_{k+1}-1)} \}
\end{equation}

The tilde on $\psi_{k+1}$ indicates the tentative representation of the sequence $\psi_{k+1}$ because the locations of prime factor $P_{k+1}$ are not represented in $\psi_{k}$.  Each sub-sequence $\psi_{k}^{(i)}$, the $i^{th}$ repetition of $\psi_{k}$, is identical in each of it's repetitions. The occurrence $\widetilde{P}$ in in the repeated sequences of $\psi_{k}$ may or may not be prime numbers, but they are the only locations where new prime numbers, i.e., $P > P_{k}$, can occur, because every other position in the sequence and its repetitions has one or more prime factors $P \le P_{k}$.  

For example:
\begin{multline}\notag
	S_{3}= \{/5,6,7,8,9,10/11,12,13,14,15,16/17,18,19,20,21,22/23,24,25,26,27,28/\\29,30,31.32.33.34/ \}
\end{multline}
The slash is used merely to indicate underlying groups of numbers corresponding to repetitions of $\psi_{2}^{(i)}$.  The underlying prime factor sequence for $S_{3}$ is then:
\begin{multline}\notag
	\psi_{3} = \{5,2\cdot3,\widetilde{P},2,3,2\cdot 5/\widetilde{P},2\cdot3,\widetilde{P},2,3\cdot5,2/\widetilde{P},2\cdot3,\widetilde{P},2\cdot5,3,2/\widetilde{P},2\cdot3,5,2,3,3/\\ \widetilde{P},2\cdot3 \cdot 5,\widetilde{P},2,3,2/ \}
\end{multline}
and $\psi_{3}$\footnote{
	One could think of $\psi_{3}$ as a 30-number elemental sequence relative to prime factors $2,3,5$ analogous to the 6-number sequence defined in Figure~\ref{T: elementalsequence}. The numbers would be represented as $N=N_{5n,i}$, where $1\le i \le 30$, with prospective prime numbers, prime to $P\le 5$ occuring at: $i=3,7,9,13,15,19,25,\; \textrm{and}\; 27$. The conversion would then be: $N_{5n,i}=5+6\cdot 5n+(i-1)$.
}
has the underlying sequence $\psi_{2}$:
\begin{multline}\notag
	\widetilde{\psi}_{3} = \{/\widetilde{P},2\cdot3,\widetilde{P},2,3,2/\widetilde{P},2\cdot3,\widetilde{P},2,3,2/\widetilde{P},2\cdot3,\widetilde{P},2,3,2/\widetilde{P},2\cdot3,\widetilde{P},2,3,3/\\ \widetilde{P},2\cdot3,\widetilde{P},2,3,2/ \}
\end{multline}
One can see from this example that the transition $\psi_{2} \rightarrow \psi_{3}$ involves repeating $\psi_{2}$ $P_{3}=5$ times in succession then include occurrences of prime factor $5$.

When we extend a sequence using repetitions of the underlying sequence of prime factors, not all numbers prime to those factors will be actual prime numbers. Therefore we refer to locations in the sequence where prime factors must appear as prospective prime numbers and when it is necessary to make that distinction we represent prospective prime numbers as $\widetilde{P}$.  Prospective prime numbers in $S_{k}$ are all numbers prime to $P \le P_{k}$. We do not consider a prime number prime to itself.  Therefore prospective prime numbers represent the only locations in $S_{k}$ where prime numbers $P > P_{k}$ can occur.

Define $S_{k}^{(i)}$ as subsets of $S_{k}$ representing each successive group of $\prod_{j=1}^{k-1}P_{j}$ numbers, which   correspond to instances of $\psi_{k-1}^{(i)}$ found in $\psi_{k}$.  Then, if we represent $S_{k}^{(i)} + N$ as adding the number $N$ to each number within $S_{k}^{(i)}$, then given $i > j$   we have:

\begin{equation}
	S_{k}^{(i)} = S_{k}^{(j)} +(i-j)\prod_{l=1}^{k-1}P_{l}
\end{equation}

Note that this relationship between $S_{k}^{(i)}$ and $S_{k}^{(j)}$ means that numbers corresponding to the same locations within $S_{k}^{(i)}$ and $S_{k}^{(j)}$ have the same prime factors from among $P \le P_{k-1}$.  That is, if $x_{n}^{(i)}$ is the $n^{th}$ number in $S_{k}^{(i)}$ and $x_{n}^{(j)}$ is the $n^{th}$ number in $S_{k}^{(j)}$ then for all $P \le P_{k-1}$: $P | x_{n}^{(i)}$ if and only if $P | x_{n}^{(j)}$. The fact that $0 \le (i-j) \le P_{k}-1$ is consistent with this rule, because $(i-j)$ can only have prime factors $P \le P_{k-1}$.

We can extend the definition of $S_{k}^{(i)}$ to successively smaller subsets of $S_{k}$ using:

\begin{align}\notag
	S_{k} &= \left\{ S_{k}^{(0)},  S_{k}^{(1)}, \cdots , S_{k}^{(P_{k}-1)} \right\} \\\notag
	&= \left\{S_{k-1}^{(0)},  S_{k-1}^{(1)}, \cdots , S_{k-1}^{(P_{k}P_{k-1}-1 )} \right\} \\\notag
	& \qquad \vdots \\\notag
	&=\left\{ S_{2}^{(0)},  S_{2}^{(1)}, \cdots , S_{2}^{( \prod_{j=3}^{k} P_{j} -1 )} \right\} \\\notag
\end{align}

Note that for $l > k$,  $S_{k} \subset S_{l}$.  

This concept of nested subsets can be applied to $\psi_{k}$ as well, however there is some loss of information since the lower level sequences represent fewer prime factors. Never-the-less it is important to realize the repetitive nature of prime factors associated with any set or subset of prime numbers. The nesting for $\psi_{k}$ looks like the following:

\begin{align}\notag
	\psi_{k} &= \left\{ \psi_{k-1}^{(0)},  \psi_{k-1}^{(1)}, \cdots , \psi_{k-1}^{(P_{k}-1)} \right\} \\\notag
	&= \left\{\psi_{k-2}^{(0)},  \psi_{k-2}^{(1)}, \cdots , \psi_{k-2}^{(P_{k}P_{k-1})-1} \right\} \\\notag
	& \qquad \vdots \\\notag
	&=\left\{ \psi_{2}^{(0)},  \psi_{2}^{(1)}, \cdots , \psi_{2}^{( \prod_{j=3}^{k} P_{j} -1)} \right\} \\\notag
\end{align}

Given this nesting of subsets, any such prime factor sequence represented by $\psi_{k} \subset \psi_{l}$, where $k <l$, continues to repeat throughout the natural numbers, e.g., as $l \rightarrow \infty$. While this is important from the fact that any combinations of prime factors produces an infinitely repeating sequence of those factors within the natural numbers, however the location of prospective prime numbers differ between the different levels of $\psi_{k}$.

We refer to $\psi_{2}=\{\widetilde{P}, 2 \cdot 3, \widetilde{P}, 2,3,2  \}$ as the elemental sequence of prime factors and correspondingly refer to $S_{2}^{(i)}$, where $0 \le i \le P_{k}-1$, as elemental sequences of any set $S_{k}$.  We can locate individual numbers precisely by their location within an elemental sequence representing them as: $N_{n,i} \in S_{2}^{(n)}$, where:

\[
S_{2}^{(n)} =\left\{N_{n,1}, N_{n,2}, N_{n,3}, N_{n,4}, N_{n,5}, N_{n,6}  \right\} 
\]

Using this sequence notation, note that prime numbers can only occur at $N_{n,1}$ or $N_{n,3}$, because all other $N_{n,i}$ have a prime factor of $2$, $3$ or both.

For our purposes we start the progressions at $N_{0,1}=5$ and $N_{0,3}=7$.  Then given that $N_{n+1,i} = N_{n,i}+6$, the numbers $N_{n,1}$ and $N_{n,3}$ represent the following progressions of numbers:

\[
N_{n,1}=6n+5 \quad \textrm{and} \quad N_{n,3} = 6n+7
\]

All prime numbers $P > 3$ therefore fall in one or the other of these two progressions although clearly, not all numbers in each progression is a prime number.

It is useful to define $P_{m,1} = N_{m,1}$ and $P_{m,3}=N_{m,3}$ when $N_{m,1}$ and $N_{m,3}$ respectively, are prime numbers. Likewise, it is also useful to define $\widetilde{P}_{m,1} = N_{m,1}$ and $\widetilde{P}_{m,3}=N_{m,3}$ when $N_{m,1}$ and $N_{m,3}$ respectively, are prospective prime numbers. 

\subsection{Propagation of Prospective Prime Numbers}\label{S: PropofProPriNumbers}

When we extend $S_{k}$ to $S_{k+1}$ we also extend $\psi_{k}$ to $\psi_{k+1}$, first by $P_{k+1}$ repetitions of $\psi_{k}$,\\ $\widetilde{\psi}_{k+1}=\left\{\psi_{k}^{(0)},\psi_{k}^{(1)}\psi_{k}^{(2)},\cdots,\psi_{k}^{(P_{k+1}-1)},\right\}$, where each $\psi_{k}^{(i)}$ represents the identical repeating sequence of prime factors $P\le P_{k}$ throughout $S_{k+1}$. Then we identify the locations of prime factor $P_{k+1}$, whereby  $\widetilde{\psi}_{k+1} \longrightarrow \psi_{k+1}$.

$S_{k}$ has $\prod_{i=1}^{k}P_{i}$ numbers: $N \in S_{k} \longrightarrow 5 \le N \le 4+\prod_{i=1}^{k}P_{i}$. When we extend $S_{k}$ to $S_{k+1}$ to include all combinations of prime factors $P \le P_{k+1}$ relative to each other we need to incude $P_{k+1}$ repetitions of $\prod_{i=1}^{k}P_{i}$ numbers.

Given

\[
n\prod_{i=1}^{k}P_{i}=\left\lfloor \frac{n\prod_{i=1}^{k}P_{i}}{P_{k+1}} \right\rfloor P_{k+1} + \left(n\prod_{i=1}^{k}P_{i}\right)\bmod{P_{k+1}}
\]

$\left(n\prod_{i=1}^{k}P_{i}\right)\bmod{P_{k+1}}=0$ if and only if $n$ is a multiple of $P_{k+1}$ and $P_{k+1}$ itself is the least such multiple.

Given any $N \in S_{k}$ and and any integer $m\ge 0$, then for all $P \le P_{k}$:
\[
P | N \leftrightarrow P | \left( N+m\prod_{i=1}^{k}P_{i} \right)
\]

Therefore the sequence of occurrence of all prime factors $P\le P_{k}$ in $S_{k}$ (i.e. $\psi_{k}$) is repeated identically in each successive sequence of $\prod_{i=1}^{k}P_{i}$ numbers.

If we limit $m$ to $0 \le m \le P_{k+1}-1$ then $\left( N+m\prod_{i=1}^{k}P_{i} \right) \in S_{k+1}$. 

Also, if we represent $S_{k+1}$ as:

\[
S_{k+1}=\left\{S_{k+1}^{(0)},S_{k+1}^{(1)},\cdots,S_{k+1}^{(P_{k+1}-1)}\right\}
\]

where each $S_{k+1}^{(i)}$ represents one of a succession of $\prod_{i=1}^{k}P_{i}$ numbers, then:

\[
S_{k+1}^{(0)}=S_{k}
\]

and for any $N\in S_{k}$ and $N'=N+m\prod_{i=1}^{k}P_{i}$, then $N' \in S_{k+1}^{(m)}$, where for all $P\le P_{k}$:

\[
P | N \longleftrightarrow P | N'
\]

Each repetition of the subsequence $\psi_{k}^{(i)}$ in $\psi_{k+1}$ represents  the fact that numbers prime to all $P \le P_{k}$ occur at the same locations within each instance of $S_{k+1}^{(i)}$.  

Numbers in $S_{k}$ that are prime to all $P \le P_{k}$ are what we have defined as prospective prime numbers in $S_{k}$. They can only occur at $N_{m,1}$ or $N_{m,3}$ in their respective elemental sequences. Those prospective prime number locations in $\psi_{k}$ are repeated $P_{k+1}$ times as an underlying sequence to $\psi_{k+1}$.  All of those repetitions of prospective prime numbers from $S_{k}$ going to $S_{k+1}$ remain prime to $P\le P_{k}$. However, the prime factor $P_{k+1}$ will occur in exactly one instance of the $P_{k+1}$ repetitions of each prospective prime number from $S_{k}$. Exactly one instance, because in the $P_{k+1}$ repetitions of $\psi_{k}$ the prime factor $P_{k+1}$ occurs exactly once in each position of $\psi_{k}$ creating a sequene of numbers of length $\prod_{i=1}^{k+1}P_{i}$ in $S_{k+1}$. 

The fact that $\frac{\prod_{i=1}^{k+1}P_{i}}{P_{k+1}}=\prod_{i=1}^{k}P_{i}$ requires that $P_{k+1}$ occurs exactly once in each position of the $\prod_{i=1}^{k}P_{i}$-number sequence spread throughout its $P_{k+1}$ repetitions in $S_{k+1}$. If $P_{k+1}$ occured as a prime factor within $S_{k+1}$ at two instances of the same repeated prospective prime number, $\widetilde{P} \in S_{k}$, i.e.,  
\[
P_{k+1} | \left(\widetilde{P}+m\prod_{i=1}^{k}P_{i}\right)\quad \textrm{and\quad} P_{k+1} | \left(\widetilde{P}+m'\prod_{i=1}^{k}P_{i}\right)
\]
where  $0 \le m,m' \le P_{k+1}-1$, then necessarily $P_{k+1} | (m'-m)$ which contradicts the condition that $m, m' < P_{k+1}$.

Also, given $N \in S_{k}$ where for some  $P \le P_{k}$ if $P | N$ then $P | \left(N+m\prod_{i=1}^{k}P_{i}\right)$, therefore, $N$ cannot be a prospective prime number when extended into $S_{k+1}$.  Therefore, all prospective prime numbers in $S_{k+1}$, i.e. numbers in $S_{k+1}$ prime to all $P\le P_{k+1}$, result from extensions of prospective prime numbers $\widetilde{P} \in S_{k}$, via $\widetilde{P}+m\prod_{i=1}^{k}P_{i}$, where $0 \le m \le P_{k+1}-1$. 

We can now state the following theorem:

\begin{theorem}\label{T: numpropinsk}  Given any prime number, $P_{k} > 2$, let $n_{k}^{\widetilde{p}}$ be the number of prospective prime numbers $\widetilde{P}$, prime relative to all $P \le P_{k}$, where $P_{k} < \widetilde{P}< 4+\prod_{i=1}^{k} P_{i}$.\footnote{Note that $4+\prod_{i=1}^{k} P_{k}$ cannot be a prime number, because it is an even number. Also, we assume that a prime number is not prime relative to itself.} Then defining $n_{1}^{\widetilde{p}}=1$: $n_{k}^{\widetilde{p}}=(P_{k}-1)n_{k-1}^{\widetilde{p}}=\prod_{i=2}^{k}(P_{i}-1)$.
\end{theorem}

\begin{proof}:  Consider the sequence of natural numbers $5 \rightarrow 4 + \prod_{i=1}^{k}P_{i}$.  This sequence consists of $\prod_{i=1}^{k}P_{i}$ sequential natural numbers that includes all combinations of prime factors  $P \le P_{k}$ situated relative to each other. This sequence of those prime factors repeats in each successive group of $\prod_{i=1}^{k}P_{i}$ natural numbers without end throughout the natural numbers. Assume this sequence has, $n_{k}^{\widetilde{p}}$ numbers prime to prime numbers $P \le P_{k}$. Then consider the extension of this sequence to the sequence of numbers $5 \longrightarrow 4+\prod_{i=1}^{k+1}P_{i}$. This involves $P_{k+1}$ repetitions of the underlying prime factor sequence associated with the numbers $5 \rightarrow  \prod_{i=1}^{k}P_{i}$. Therefore there are $P_{k}n_{k}^{\widetilde{p}}$ numbers prime to $P<P_{k}$ in the extended sequence. Since the prime factor $P_{k+1}$ occurs exactly once in each position of the underlying $\prod_{i=1}^{k}P_{i}$ prime factor sequence spread throughout its $P_{k+1}$ repetitions, each of the numbers prime to $P \le P_{k}$ has the prime factor  $P_{k+1}$ exactly once during the $P_{k+1}$ iterations of the underlying sequence. Therefore there are $n_{k+1}^{\widetilde{p}}=P_{k+1}n_{k}^{\widetilde{p}} - n_{k}^{\widetilde{p}}=(P_{k}-1)n_{k}^{\widetilde{p}}$ numbers prime to all $P \le P_{k+1}$ within the sequence $5 \rightarrow 4 + \prod_{i=1}^{k+1}P_{i}$. If then we define $n_{1}^{\widetilde{p}}=1$, to initialize the succession of prime factor sequences, we have: $n_{k}^{\widetilde{p}}=\prod_{i=2}^{k}(P_{i}-1)$. Starting with $k=2$ we have the sequence $5, 6, 7, 8, 9, 10$ where only $5$ and $7$ are prime to $2$ and $3$ giving $n_{2}^{\widetilde{p}}=2$. Then extending this to include prime factor $5$ we get the sequence $5, 6, \cdots, 34$, where only $7, 11, 13, 17, 19, 23, 29,$ and $31$ are prime to $2, 3,$ and $5$, giving $n_{3}^{\widetilde{p}}=8$, which verifies the initial cases of the formula $n_{k}^{\widetilde{p}}=\prod_{i=2}^{k}(P_{i}-1)$ and completes the proof.
\end{proof}

We see from Theorem~\ref{T: numpropinsk} that $n_{k}^{\widetilde{p}}$ is the number of prospective prime numbers in $S_{k}=\left\{5\longrightarrow 4+\prod_{i=1}^{k}P_{i}\right\}$. That is the number of numbers in positions $N_{n,1}$ and $N_{n,3}$ in $S_{k}$ that are prime to all $P \le P_{k}$.

Note that $n_{k}^{\widetilde{p}}=\prod_{i=2}^{k}(P_{i}-1)$ alternatively results from using the number $N=\prod_{i=1}^{k}P_{i}$ in Euler's totient function:\cite{FG}
\[
\varphi(N)=N\prod_{P|N}\left(1-\frac{1}{P}\right) \longrightarrow \prod_{i=1}^{k}P_{i}\prod_{i=1}^{k}\left(1-\frac{1}{P_{i}}\right) = \prod_{i=1}^{k}P_{i}\left(1-\frac{1}{P_{i}}\right)=\prod_{i=2}^{k} (P_{i} -1 )
\]
which has he same interpretation.

The average density of prospective prime numbers among all numbers in the set of natural numbers $S_{k}$ is :

\begin{equation}\label{E: densityPPN-1}
	\rho^{\widetilde{P}}_{k} = \frac{\prod_{i=1}^{k} (P_{i} -1 )}{\prod_{j=1}^{k}P_{j}} = \prod_{i=1}^{k} \left(1-\frac{1}{P_{k}} \right)
\end{equation}

This can be viewed as a partial or finite reciprocal of Euler's product for the Riemann Zeta function with $s=1$:

\begin{equation}\label{E: Zetak}
	\zeta_{k}(s=1)=\prod_{i=1}^{k} \left(1-\frac{1}{P_{k}} \right)^{-1}
\end{equation}

This reciprocal version of Equation~(\ref{E: densityPPN-1}) represents the average number of natural numbers per prospective prime number in $S_{k}$ or the average spacing of prospective prime numbers among all numbers in $S_{k}$.

Letting $k \rightarrow \infty$ we get the full Riemann Zeta function for $s=1$ which diverges:

\begin{equation}\label{E: Zeta}
	\zeta(1)=\prod_{i=1}^{\infty} \left(1-\frac{1}{P_{k}} \right)^{-1} \longrightarrow \infty
\end{equation}
which means, like with actual prime numbers, the average density of prospective prime numbers becomes vanishingly small in $S_{k}$ as $k \rightarrow \infty$.

Next consider the fraction of prospective prime numbers in $S_{k}$ that are actual prime numbers:

\begin{equation}\label{E: prime2proprimeratio}
	\boldsymbol{\pi}_{k}^{\widetilde{P}}=\frac{\boldsymbol{\pi}\left(4+\prod_{i=1}^{k}P_{i}\right)-k}{n_{k}^{\widetilde{P}}}
\end{equation}

$k$ appears in the numerator because $P \le P_{k}$ are not considered as prospective prime numbers in $S_{k}$, because they are already established prime numbers within this concept and are not considered prime to themselves. 

Table~\ref{T: numericalevalpsi} gives numerical examples of Equation~(\ref{E: prime2proprimeratio}) for $k = 2 \rightarrow 10$.

\begin{table}[h]
	\begin{center}
		\caption{Numerical evaluation of $\boldsymbol{\pi}_{k}^{\widetilde{P}}$ for small $k$.} 
		\renewcommand{\arraystretch}{1.5}
		\begin{tabular}{|c|c|c|c|}
			\hline 
			$k$ & $\boldsymbol{\pi}\left(4+\prod_{i=1}^{k}P_{i}\right)-k$ & $n_{k}^{\widetilde{P}}$ & $\boldsymbol{\pi}_{k}^{\widetilde{P}}$ \\ 
			\hline 
			2 & $\boldsymbol{\pi}(10)-2=2$ & $\prod_{i=1}^{2}(P_{i}-1)=2$ & 1 \\ 
			\hline 
			3 & $\boldsymbol{\pi}(34)-3=8$ & $\prod_{i=1}^{3}(P_{i}-1)=8$ & 1 \\ 
			\hline 
			4 & $\boldsymbol{\pi}(214)-4=43$ & $\prod_{i=1}^{4}(P_{i}-1)=48$ & .896 \\ 
			\hline 
			5 & $\boldsymbol{\pi}(2314)-5=339$ & $\prod_{i=1}^{5}(P_{i}-1)=480$ & .706 \\ 
			\hline 
			6 & $\boldsymbol{\pi}(30,034)-6=3,242 $ & $\prod_{i=1}^{6}(P_{i}-1)=5,760$ &.563  \\ 
			\hline 
			7 & $\boldsymbol{\pi}(510,514)-7=42,204 $ & $\prod_{i=1}^{7}(P_{i}-1)=92,160$ & .458 \\ 
			\hline 
			8 & $\boldsymbol{\pi}(9,699,694)-8=646,021 $ & $\prod_{i=1}^{8}(P_{i}-1)=1,658,880$ & .389 \\ 
			\hline 
			9 & $\boldsymbol{\pi}(223,092,874)-9=12,283,522 $ & $\prod_{i=1}^{9}(P_9{i}-1)=36,495,360$ &.337  \\ 
			\hline 
			10 & $\boldsymbol{\pi}(6,469,693,234)-10=300,369,786 $ & $\prod_{i=1}^{10}(P_{i}-1)=1,021,870,080$ & .294 \\ 
			\hline  
		\end{tabular}
		\label{T: numericalevalpsi}
	\end{center}
\end{table}

In order to evaluate the trend in $\boldsymbol{\pi}_{k}^{\widetilde{P}}$ for large $k$ we use the approximation $\boldsymbol{\pi}(N)\approx\frac{N}{\ln{N}}$ in Equation~(\ref{E: prime2proprimeratio}) giving:

\begin{align}\label{E: fracprimes}\notag
	\boldsymbol{\pi}_{k}^{\widetilde{P}}\approx& \left[\frac{4+\prod_{i=1}^{k}P_{i}}{\ln{\left(4+\prod_{i=1}^{k}P_{i}\right)}}-k\right]\cdot \frac{1}{\prod_{i=1}^{k}(P_{i}-1)} \\
	\approx & \frac{1}{\ln{\left(4+\prod_{i=1}^{k}P_{i}\right)}}\cdot 
	\left[\frac{ 4+\prod_{i=1}^{k}P_{i}}{\prod_{i=1}^{k}(P_{i}-1)} \right]-\frac{k}{\prod_{i=1}^{k}(P_{i}-1)}
\end{align}

Ignoring the 4 relative to the product term gives:

\begin{align}\label{E: fracprimes2}
	\boldsymbol{\pi}_{k}^{\widetilde{P}}
	\approx & \frac{1}{\sum_{i=1}^{k}\ln{P_{i}}}\cdot \prod_{i=1}^{k}\left(1-\frac{1}{P_{i}}\right)^{-1}-\frac{k}{\prod_{i=1}^{k}(P_{i}-1)}
\end{align}

Given that prime numbers have no upper bound,  $\boldsymbol{\pi}_{k}^{\widetilde{P}}$ must remain positive as $k\rightarrow \infty$ and the negative term clearly tends toward zero. It turns out the negative term decreases much faster than the first term and therefore can be ignored. Ignoring the negative term is equivalent to redefining $\boldsymbol{\pi}_{k}^{\widetilde{P}}$ as the ratio of all prime numbers in $S_{k}$, i.e. including $P\le P_{k}$, to all prospective prime numbers in $S_{k}$. However, the number of prospective prime numbers grows multiplicative while the number of prime numbers less than or equal to $P_{k}$, i.e. $k$, grows linearly and therefore $P\le P_{k}$ can be ignored.

The reciprocal summation term, which is equal to $\ln{\left(\prod_{i=1}^{k}P_{i}\right)}$, tends to zero while the multiplication term tends to infinity as $k\rightarrow \infty$. Ignoring the negative term, look at the ratio:

\[
\frac{\boldsymbol{\pi}_{k+1}^{\widetilde{P}}}{\boldsymbol{\pi}_{k}^{\widetilde{P}}}\approx
\left(\frac{P_{k+1}}{P_{k+1}-1}\right)\cdot \left[\frac{1}{1+\frac{\ln{P_{k+1}}}{\sum_{i=1}^{k}\ln{P_{i}}}}\right]
\]

The fraction in the denominator in the square brackets is clearly less than 1, so let $\frac{\ln{P_{k+1}}}{\sum_{i=1}^{k}\ln{P_{i}}}=\frac{1}{\alpha_{k}}$ where $\alpha_{k}>1$, giving:

\begin{equation}\label{E: fixedratiopri2propri}
	\frac{\boldsymbol{\pi}_{k+1}^{\widetilde{P}}}{\boldsymbol{\pi}_{k}^{\widetilde{P}}}\approx
	\left(\frac{P_{k+1}}{P_{k+1}-1}\right)\cdot \left(\frac{\alpha_{k}}{\alpha_{k}+1}\right)
\end{equation}

Where:

\[
\alpha_{k}\ln{P_{k+1}}=\sum_{i=1}^{k}\ln{P_{i}}=\ln{\left(\prod_{i=1}^{k}P_{i}\right)} \longrightarrow P_{k+1}^{\alpha_{k}}=\prod_{i=1}^{k}P_{i}
\]

In Equation~(\ref{E: fixedratiopri2propri}) the first fractional factor is asymtotic to 1 from above while the second fractional factor is asymtotic to 1 from below and their product is asymtotic to 1 from below (see Table~\ref{T: P2proPratio}).  Therefore we always have: $\boldsymbol{\pi}_{k+1}^{\widetilde{P}}<\boldsymbol{\pi}_{k}^{\widetilde{P}}$, while $\frac{\boldsymbol{\pi}_{k+1}^{\widetilde{P}}}{\boldsymbol{\pi}_{k}^{\widetilde{P}}}\longrightarrow 1$ as $k\rightarrow \infty$.

As an example consider: $P_{9591}=99,989$, where $\boldsymbol{\pi}_{9591}^{\widetilde{P}}=.000102892$ and  $P_{9592}=99,991$, where $\boldsymbol{\pi}_{9592}^{\widetilde{P}}=.000102882$, giving $\frac{\boldsymbol{\pi}_{9592}^{\widetilde{P}}}{\boldsymbol{\pi}_{9591}^{\widetilde{P}}}=.999902811$.

If we then look at $P_{11176}=118633$ where $\boldsymbol{\pi}_{11176}^{\widetilde{P}}=.0000881401$ and $P_{11177}=118661$ where $\boldsymbol{\pi}_{11177}^{\widetilde{P}}=.0000881322$ giving $\frac{\boldsymbol{\pi}_{11177}^{\widetilde{P}}}{\boldsymbol{\pi}_{11176}^{\widetilde{P}}}=.9999103700$.

$\boldsymbol{\pi}_{\widetilde{P}}^{k}$ appears to asymtotically approach zero as $k\rightarrow \infty$, where:

\[\sum_{i=1}^{k}\ln{P_{i}}>\prod_{i=1}^{k}\left(1-\frac{1}{P_{i}}\right)^{-1}=\prod_{i=1}^{k}\left(\frac{P_{i}}{P_{i}-1}\right)
\]

\begin{table}[h]
	\begin{center}
		\caption{Evaluation of Equation~(\ref{E: fixedratiopri2propri})}
		\renewcommand{\arraystretch}{1.5}	
		\begin{tabular}{|c|c|c|c|c|c|c|c|}
			\hline 
			$k$ & $P_{k}$ & $P_{k+1}$ & $\frac{P_{k+1}}{P_{k+1}-1}$ & $\prod_{i=1}^{k}P_{i}$ & $\alpha_{k}$ & $\frac{\alpha_{k}}{\alpha_{k}+1}$ & $\frac{\boldsymbol{\pi}_{k+1}^{\widetilde{P}}}{\boldsymbol{\pi}_{k}^{\widetilde{P}}}$ \\ 
			\hline 
			2 & 3 & 5 & 1.25 & 6 & 1.11 & .526 & .658 \\ 
			\hline 
			3 & 5 & 7 & 1.16 & 30 & 1.75 & .636 & .742 \\ 
			\hline 
			4 & 7 & 11 & 1.1 & 210 & 2.23 & .690 & .759 \\ 
			\hline 
			5 & 11 & 13 & 1.083 & 2,310 & 3.02 & .751 & .814 \\ 
			\hline 
			6 & 13 & 17 & 1.062 & 30,030 & 3.64 & .784 & .833 \\ 
			\hline 
			7 & 17 & 19 & 1.05 & 510,510 & 4.46 & .816 & .862 \\ 
			\hline 
			8 & 19 & 23 & 1.045 & 9,699,690 & 5.13 & .837 & .875 \\ 
			\hline 
			9 & 23 & 29 & 1.035 & 22,3092,870 & 5.71 & .851 & .881 \\ 
			\hline 
		\end{tabular}\\
		\label{T: P2proPratio} 	
	\end{center}
\end{table}

\subsection{Propagation of Prospective twin primes}

Prospective twin primes, where $\widetilde{P}_{i+1} = \widetilde{P}_{i} + 2$ follow the same propagation concept as prospective prime numbers except they are paired occurrences within the same elemental sequence, appearing as $\widetilde{P_{i}}=\widetilde{P}_{m,1}$ and $\widetilde{P}_{i+1} = \widetilde{P}_{m,3}$. 

When extending $S_{k}$ to $S_{k+1}$ each prospective twin prime $(\widetilde{P}_{m,1},\widetilde{P}_{m,3}) \in S_{k}$, prime relative to $P \le P_{k}$ corresponds to a location of paired prospective prime number locations in the underlying sequence $\psi_{k}$. That sequence is repeated $P_{k+1}$ times in $S_{k+1}$ and each instance retains that paired set of locations corresponding to numbers in $S_{k+1}$ that remain prime to all $P \le P_{k}$.

These repetitions in $S_{k+1}$ occur as:

\begin{equation}\label{E: tppProp1}
	\left(\widetilde{P}_{i}+m_{k}\prod_{j=1}^{k}P_{j} \; , \widetilde{P}_{i+1}+m_{k}\prod_{j=1}^{k}P_{j}\right)  =  \left(\widetilde{P}_{[m+ m_{k}\prod_{j=3}^{k}P_{j} ],1} \; , \widetilde{P}_{[m + m_{k}\prod_{j=3}^{k}P_{j} ],3}\right)
\end{equation}

As with the propagation of prospective prime numbers, the prime factor $P_{k+1}$ occurs in each position of $\psi_{k}$ exactly once throughout its $P_{k+1}$ repetitions in $S_{k+1}$ and it cannot occur twice in the same elemental sequence. Therefore, one instance of prime factor $P_{k+1}$ occurs in two of the repeated locations of each prospective twin prime in $S_{k}$ extended to $S_{k+1}$, corresponding to the two components of the prospective twin primes.  

This leads to the following theorem, using $\widetilde{t}$ to represent a prospective twin prime:

\begin{theorem}\label{T: noprotpp}  Given any prime number, $P_{k} > 2$, let $n_{k}^{\widetilde{t}}$ be the number of prospective twin primes $(\widetilde{P}_{i}, \widetilde{P}_{i+1})$, where $\widetilde{P}_{i+1}=\widetilde{P}_{i}+2$ and  $\widetilde{P}_{i}$ and $\widetilde{P}_{i+1}$ are prime relative to all $P \le P_{k}$, and where $P_{k} < \widetilde{P}_{i}$ and $\widetilde{P}_{i+1} < 4+\prod_{i=1}^{k} P_{k}$. Then defining $n_{1}^{\widetilde{t}}=1$,: $n_{k}^{\widetilde{t}}=(P_{k}-2)n_{k-1}^{\widetilde{t}}=\prod_{i=2}^{k}(P_{i}-2)$.
\end{theorem}

\begin{proof}
	Consider the sequence of natural numbers $5 \rightarrow 4 + \prod_{i=1}^{k-1}P_{i}$.  This sequence consists of $\prod_{i=1}^{k-1}P_{i}$ sequential natural numbers that includes all combinations of prime factors  $P<P_{k-1}$ situated relative to each other. That sequence of prime factors repeats without end throughout the natural numbers for each successive sequence of $\prod_{i=1}^{k-1}P_{i}$ numbers. Assume this sequence has, $n_{k-1}$ prospective twin primes prime to all prime numbers $P \le P_{k-1}$.  Then consider the extension of this sequence to the number $4+\prod_{i=1}^{k}P_{i}$. This involves $P_{k}$ repetitions of the underlying prime factor sequence associated with the numbers $5 \rightarrow  \prod_{i=1}^{k-1}P_{i}$. Therefore there are $P_{k}n_{k-1}$ prospective twin primes prime to $P<P_{k-1}$ in the extended sequence.  Since the prime factor $P_{k}$ occurs exactly once in each of the $\prod_{i=1}^{k-1}P_{i}$ positions of the underlying  prime factor sequence, each of the twin primes prime to $P \le P_{k-1}$ has the prime factor  $P_{k}$ exactly twice during the $P_{k}$ iterations of the underlying sequence, once for each number in each twin prime,  and the prime factor $P_{k}$ never occurs more than once in the same 6-number elemental sequence. Therefore there are $n_{k}=P_{k}n_{k-1} - 2n_{k-1}=(P_{k}-2)n_{k-1}$ twin primes prime to all $P \le P_{k}$ within the sequence $5 \rightarrow 4 + \prod_{i=1}^{k}P_{i}$. If then we define $n_{1}=1$, to initialize the succession of prime factor sequences, we have: $n_{k}=\prod_{i=2}^{k}(P_{i}-2)=\prod_{i=3}^{k}(P_{i}-2)$. 
	
	Starting with $k=2$ we have the sequence $5, 6, 7, 8, 9, 10$ where $5$ and $7$ are the only twin prime prime to $2$ and $3$ giving $n_{2}^{\widetilde{t}}=1$. Then extending this to include prime factor $5$ we get the sequence $5, 6, \cdots, 34$, where only $ (11, 13), (17, 19),$ and $(29,31)$ are twin primes prime to $2, 3,$ and $5$, giving $n_{3}^{\widetilde{t}}=3$, which  verifies the initial cases of the formula $n_{k}^{\widetilde{t}}=\prod_{i=2}^{k}(P_{i}-2)$ and completes the proof.
\end{proof}

The theorem shows the number of prospective twin primes in 

$S_{k}=\left\{N: 5 \le N \le 4+\prod_{i=1}^{k}P_{i}\right\}$ is given by:

\begin{equation}\label{E: numbertpp_1} 
	n_{k}^{\widetilde{t}} = \prod_{i=2}^{k}(P_{i} - 2)
\end{equation}

The density of prospective twin primes relative to all numbers in $S_{k}$ is:

\begin{equation}\label{E: densityPTPP1}
	\rho_{k}^{\widetilde{t}} =\frac{\prod_{i=2}^{k}(P_{i} - 2)}{\prod_{j=1}^{k}P_{j}}=\frac{1}{2}\prod_{i=2}^{k}\frac{P_{i}-2}{P_{i}} = \frac{1}{2} \prod_{i=2}^{k} \left( 1-\frac{2}{P_{i}} \right)
\end{equation}

The density of prospective twin primes relative to prospective prime numbers in $S_{k}$ is:

\begin{equation}\label{E: densityPTPP2}
	\sigma_{k}^{\widetilde{t}}=\frac{\prod_{i=2}^{k}\left(P_{i}-2\right)}{\prod_{i=2}^{k}\left(P_{i}-1\right)}=\prod_{i=2}^{k}\frac{P_{i}-2}{P_{i}-1}
\end{equation}

\begin{table}
	\begin{center}
		\caption{Numbers ($n_{k}^{\widetilde{t}}$) of prospective twin primes and their associated densities with respects to all numbers ($\rho_{k}^{\widetilde{t}}$) in a set $S_{k}$ and with respects to prospective prime numbers ($\sigma_{k}^{\widetilde{t}}$) in that set.}
		\renewcommand{\arraystretch}{1.5}
		\begin{tabular}{|c|c|c|c|c|c|}
			\hline 
			$k$ & $P_{k}$ & $4+\prod_{i=1}^{k}P_{i}$ & $n_{k}^{\widetilde{t}}$ & $\rho_{k}^{\widetilde{t}}$ & $\sigma_{k}^{\widetilde{t}}$ \\ 
			\hline 
			3 & 5 & 34 & 3 & .1 & .375 \\ 
			\hline 
			4 & 7 & 214 & 15 & .0714 & .313 \\ 
			\hline 
			5 & 11 & 2,314 & 135 & .05834 & .281 \\ 
			\hline 
			6 & 13 & 30,034 & 1,485 & .0495 & .258 \\ 
			\hline 
			7 & 17 & 510,514 & 22,275 & .0436 & .242 \\ 
			\hline 
			8 & 19 & 9,699,694 & 378,675 & .0390 & .228 \\ 
			\hline 
			9 & 23 & 223,092,874 & 7,952,175 & .0356 & .218 \\ 
			\hline 
			10 & 29 & 6,469,693,234 & 214,708,725 & .0332 & .210 \\ 
			\hline 
			11 & 31 & 200,560,490,134 & 6,226,553,035 & .0310 & .203 \\ 
			\hline 
			12 & 37 & 7,420,738,134,814 & 217,929,355,875 & .0294 & .197 \\ 
			\hline 
		\end{tabular} 
		\label{T: primedensities}
	\end{center}
\end{table}

Given that all prospective prime numbers $\widetilde{P} \in S_{k}$ occur in the interval, $P_{k} < \widetilde{P} < 4+\prod_{i=1}^{k}P_{i}$, one might think using that interval would give a more accurate density for prospective prime numbers and prospective twin primes, however:

\[
\widehat{n}_{k}^{\widetilde{p}}=
\frac{\prod_{i=1}^{k}(P_{i}-1)}{\prod_{i=1}^{k}P_{i}-P_{k}}=
\left(\frac{1}{1-\frac{1}{\prod_{i=1}^{k-1}P_{i}}}\right)
\frac{\prod_{i=1}^{k}(P_{i}-1)}{\prod_{i=1}^{k}P_{i}} \approx n_{k}^{\widetilde{p}}
\]

and

\[
\widehat{n}_{k}^{\widetilde{t}}=
\frac{\prod_{i=1}^{k}(P_{i}-2)}{\prod_{i=1}^{k}P_{i}-P_{k}}=
\left(\frac{1}{1-\frac{1}{\prod_{i=1}^{k-1}P_{i}}}\right)
\frac{\prod_{i=1}^{k}(P_{i}-2)}{\prod_{i=1}^{k}P_{i}} \approx n_{k}^{\widetilde{t}}
\]

Given Theorem~\ref{T: noprotpp} we also have the following theorem.

\begin{theorem}
	Given any prime number $P_{k}$, there are an infinite number of prospective twin primes, prime to $P\le P_{k}$
\end{theorem}  

\begin{proof}
	Given Theorem~\ref{T: noprotpp}, there are $n_{k}^{\widetilde{t}}=\prod_{i=3}^{k}(P_{i}-2)$ prospective twin primes relative to prime factors $P \le P_{k}$ in the sequence $5 \rightarrow 4 + \prod_{i=1}^{k}P_{i}$. Since this sequence of natural numbers includes all possible combinations of prime factors $P \le P_{k}$ situated relative to each other, the underlying sequence of prime factors $P \le P_{k}$ repeats without end throughout the natural numbers. 
	
	Then if $(\widetilde{P}_{i},\widetilde{P}_{i+1})\in S_{k}$ is such a prospective twin prime relative to all $P \le P_{k}$, then so is: 
	
	\[
	\left( \widetilde{P}_{i}+ n\prod_{i=1}^{k}P_{i} \;, \widetilde{P}_{i+1} + n\prod_{i=1}^{k}P_{i}\right)
	\]
	
	where $n$ is any positive integer.
	
	Therefore given that $n_{k}^{\widetilde{t}}$ is not zero for any $P_{k}$, there are an infinite number of prospective twin primes relative to $P \le P_{k}$.
\end{proof}

\subsection{Systematic Generation of Prospective Prime Numbers and Prospective twin primes}

The initial elemental 6-number sequence contains prime numbers $N_{0,1}=5$ and $N_{0,3}=7$. Then numbers that remain prime to $2$ and $3$ occur at:

\begin{equation}\label{E: sysmethod 1}
	N_{0,1}+m\prod_{i=1}^{2}P_{i} \quad \textrm{and} \quad  N_{0,3}+m\prod_{i=1}^{2}P_{i}
\end{equation}
where $m$ has any integer value.

Considering the succession of sets $S_{2} \rightarrow S_{3} \rightarrow S_{3} \rightarrow \cdots \rightarrow S_{k}$, we can write Equation~(\ref{E: sysmethod 1}) as:\\

For $N_{n,1}\; \& \; N_{n,3} \in S_{k}$:

\begin{align}\label{E: proprime1a}\notag
	\left(\begin{array}{c} N_{n_{1},1}\\ N_{n_{3},3}\end{array}\right)_{S_{k}}
	&=  
	\left(\begin{array}{c} N_{0,1}\\ N_{0,3}\end{array}\right)
	+\left(\begin{array}{c} m_{3}^{(1)}\\ m_{3}^{(3)}\end{array}\right)\prod_{i=1}^{2}P_{i}+\left(\begin{array}{c} m_{4}^{(1)}\\ m_{4}^{(3)}\end{array}\right)\prod_{i=1}^{3}P_{i}+\cdots \\ \notag
	&+ \sum_{j=3}^{k}\left(\begin{array}{c} m_{j}^{(1)}\\ m_{j}^{(3)}\end{array}\right)\prod_{i=1}^{j-1}P_{i} \\
	\left(\begin{array}{c} N_{n_{1},1}\\ N_{n_{3},3}\end{array}\right)_{S_{k}}	&=
	\left(\begin{array}{c} N_{0,1}\\ N_{0,3}\end{array}\right)+\sum_{j=3}^{k}\left(\begin{array}{c} m_{j}^{(1)}\\ m_{j}^{(3)}\end{array}\right)\prod_{i=1}^{j-1}P_{i}
\end{align}
Where the upper numbers go together and the lower numbers go together representing two separate equations.

Also: $0\le m_{j} \le P_{j}-1$ at each stage and:

\begin{equation}\label{E: proprime1c}\notag
	n_{1}=\sum_{j=3}^{k} m_{j}^{(1)}\prod_{i=3}^{j-1}P_{i} \quad
	\textrm{and} \quad
	n_{3}=\sum_{j=3}^{k} m_{j}^{(3)}\prod_{i=3}^{j-1}P_{i}
\end{equation}

Numbers prime to $P \le P_{k}$ in $S_{k}$ are produced using this equation by including terms up to and including $m_{k}$, while noting that one value of $m_{j}$ in each successive stage of the summation must be omitted separately for each of $N_{n_{1},1}$ and $N_{n_{3},3}$ because those cases produce numbers divisible by $P_{j}$. The reason only one value is omitted for each set of $m_{j}$ is because when extending $S_{j-1}$ to $S_{j}$ the underlying sequence $\psi_{j-1}$ is repeated $P_{j}$ times in $\psi_{j}$ and the prime factor $P_{j}$ occurs exactly once in each position of $\psi_{j-1}$ spread throughout it's $P_{j}$ repetitions in $\psi_{j}$. Therefore exactly one instance of $N_{n_{1},1}$ and $N_{n_{3},3}$ for each $m_{j}$ will have prime factor $P_{j}$ and those correspond to different values of $m_{j}$ in the case of $N_{n_{1},1}$ and $N_{n_{3},3}$, because the prime factor $P_{j}>6$ cannot occur twice in the same elemental sequence.

Therefore the supplemental condition for Equations~(\ref{E: proprime1a})  is that for each successive stage of summation:

\begin{equation}\label{E: supplcondition}
	P_{l} \nmid \left[ \left(\begin{array}{c} N_{0,1}\\ N_{0,3}\end{array}\right)+\sum_{j=3}^{l}\left(\begin{array}{c} m_{j}^{(1)}\\ m_{j}^{(3)}\end{array}\right)\prod_{i=1}^{j-1}P_{i} \right]
\end{equation}

This condition can also be stated as:

\[
\left[ \left(\begin{array}{c} N_{0,1}\\ N_{0,3}\end{array}\right)+\sum_{j=3}^{l}\left(\begin{array}{c} m_{j}^{(1)}\\ m_{j}^{(3)}\end{array}\right)\prod_{i=1}^{j-1}P_{i} \right]\bmod{P_{l}}\ne 0
\]

Starting with prospective prime numbers in $S_{k}$ and extending them to $S_{k+1}$, using all values of $m_{k+1}$ produces all numbers in $S_{k+1}$ that remain prime to $P \le P_{k}$. Elimiating the one value of $m_{k+1}$ for each $\widetilde{P}\in S_{k}$ then gives the prospective prime numbers in $S_{k+1}$ that are prime to $P \le P_{k+1}$.

Numbers in $S_{k}$ that have a prime factor from among $P \le P_{k}$ cannot be extended into $S_{k+1}$ as a prospective prime number because, given $N \in S_{k}$, if $P | N$ then $P | \left(N+m_{k+1}\prod_{i=1}^{k}P_{i}\right)$ for all values of $m_{k+1}$. That is why the existence of prospective prime numbers $\widetilde{P}>P_{k}$ in $S_{k+1}$ depend on extension from prospective prime numbers in $S_{k}$ and actual prime numbers $P>P_{k}$ can only occur at locations of prospective prime numbers and in this respect they also depend on extention from the prospective prime numbers in $S_{k}$. 

Calculation of prospective prime numbers is best done iteratively, formulating Equation~(\ref{E: proprime1a}) for $S_{k}\longrightarrow S_{k+1}$ as:

\begin{equation}\label{E: iteritiveeq}
	\left(\begin{array}{c}\widetilde{P}_{n_{1},1}\\\widetilde{P}_{n_{3},3}\end{array} \right)_{S_{k+1}}=\left(\begin{array}{c}\widetilde{P}_{n'_{1},1}\\\widetilde{P}_{n'_{3},3}\end{array} \right)_{S_{k}}+\left(\begin{array}{c} m_{k+1}^{(1)}\\m_{k+1}^{(3)}\end{array}\right)\prod_{i=1}^{k}P_{i}
\end{equation}
where $0 \le m_{k+1}^{(1)},m_{k+1}^{(3)} \le P_{k+1}-1$ and the condition from equation~(\ref{E: supplcondition}) gives: $m_{k+1}^{(1)}\ne \widehat{m}_{k+1}^{(1)}$ and $m_{k+1}^{(3)}\ne \widehat{m}_{k+1}^{(3)}$, where $\widehat{m}_{k+1}^{(1)}$ and $\widehat{m}_{k+1}^{(3)}$ are disallowed values (since they give results divisible by $P_{k+1}$) given by:

\begin{align}\label{E: mhat}\notag
	\widehat{m}_{k+1}^{(1)}=&\frac{\alpha_{1}P_{k+1}-\widetilde{P}_{n'_{1},1}\bmod{P_{k+1}}}{\left(\prod_{i=1}^{k}P_{i}\right)\bmod{P_{k+1}}}\\	
	\textrm{and}\\\notag
	\widehat{m}_{k+1}^{(3)}=&\frac{\alpha_{3}P_{k+1}-\widetilde{P}_{n'_{3},3}\bmod{P_{k+1}}}{\left(\prod_{i=1}^{k}P_{i}\right)\bmod{P_{k+1}}}	
\end{align}
where $\alpha_{1}$ and $\alpha_{3}$ are the smallest integers making $\widehat{m}_{k+1}^{(1)}$ and $\widehat{m}_{k+1}^{(3)}$ integers; and $\alpha_{1}=0$ if $\widetilde{P}_{n'_{1},1}\bmod{P_{k+1}}=0$ and $\alpha_{3}=0$ if $\widetilde{P}_{n'_{3},1}\bmod{P_{k+1}}=0$.

Note that using the maximum value of $m_{j}=P_{j}-1$ at each stage in Equations~(\ref{E: proprime1a}) gives:

\begin{align}\label{E: maxmj-1}\notag
	\left(\begin{array}{c} N_{n_{1},1}\\ N_{n_{3},3}\end{array}\right)_{S_{k}}
	&=  
	\left(\begin{array}{c} N_{0,1}\\ N_{0,3}\end{array}\right)+\sum_{j=3}^{k}(P_{j}-1)\prod_{i=1}^{j-1}P_{i}\\\notag
	&=\left(\begin{array}{c} N_{0,1}\\ N_{0,3}\end{array}\right)+
	\sum_{j=3}^{k}\prod_{i=1}^{j}P_{i}-\sum_{j=3}^{k}\prod_{i=1}^{j-1}P_{i}\\\notag
	&=\left(\begin{array}{c} 5\\ 7 \end{array}\right)+\prod_{i=1}^{k}P_{i}-6\\
	&=\prod_{i=1}^{k}P_{i} + \left(\begin{array}{c} -1\\ +1 \end{array}\right)
\end{align}

$\prod_{i=1}^{k}P_{i}-1$ and $\prod_{i=1}^{k}P_{i}+1$ are prime to $P \le P_{k}$. they are therefore largest prospective prime numbers and prospective twin prime in $S_{k}$.\footnote{However, they are not alway an actual prime numbers, e.g., $\prod_{i=1}^{4}P_{i}-1=209=11\cdot 19$; and $\prod_{i=1}^{6}P_{i}+1=30031=59\cdot 509$.} Therefore using the maximum value of $m_{j}$ at every stage is always a valid choice since it always satisfies the supplemental condition (\ref{E: supplcondition}). However, the maximum value of $m_{j}$ may not be valid if a value less than the maximum value was used at a prior stage. 

In extending a prospective prime number $\widetilde{P} \in S_{k}$ to $S_{k+1}$, the choice of $m_{j}=0$ is valid unless $\widetilde{P}$ itself contains the prime factor $P_{k+1}$.

We can now state the following theorem:

\begin{theorem}
	Given $S_{k}=\left\{N: 5 \le N \le 4 + \prod_{i=1}^{k}P_{i} \right\}$, then let $\widetilde{P}_{n,1}$ and $\widetilde{P}_{n',3}$ represent numbers in $S_{k}$ from the arithmetic progressions $6n+5$ and $6n'+7$ respectively ($n \ge 0$) that are also prime to all $P \le P_{k}$. Then there are equal numbers of $\widetilde{P}_{n,1}$ and $\widetilde{P}_{n',3}$ in $S_{k}$.
\end{theorem}

\begin{proof}
	This follows from Equation~(\ref{E: proprime1a}), because in each summation term there are equal numbers of valid $m_{j}$ for both $N_{n,1}$ and $N_{n',3}$. This is true even when considering the supplemental condition Equation~(\ref{E: supplcondition}) which eliminates one value of $m_{j}$ for each of  $N_{n,1}$ and $N_{n',3}$ at each summation stage. Equation~(\ref{E: proprime1a}) together with its supplemental condition yields numbers in $S_{k}$ that are prime to all $P \le P_{k}$.
\end{proof}

The generation of prospective twin primes follows the same scheme except at each stage we must have:

\[
m_{j}^{(1)}=m_{j}^{(3)}=m_{j}
\]
where both $m_{j} \ne \widehat{m}_{j}^{(1)}$ and  $m_{j}\ne \widehat{m}_{j}^{(3)}$ as defined by Equation~(\ref{E: mhat}).

Therefore for prospective twin primes, Equation~(\ref{E: proprime1a}) becomes:

\begin{equation}\label{E: proprime1a3}
	\left(\begin{array}{c} N_{n,1}\\ N_{n,3}\end{array}\right)_{S_{k}}
	=
	\left(\begin{array}{c} N_{0,1}\\ N_{0,3}\end{array}\right)+\sum_{j=3}^{k}(m_{j})\prod_{i=1}^{j-1}P_{i}
\end{equation}
which we can abbreviate as:

\begin{equation}\label{E: proprime1a2}
	\widetilde{t}_{k}
	=
	\left(\begin{array}{c} N_{0,1}\\ N_{0,3}\end{array}\right)+\sum_{j=3}^{k}(m_{j})\prod_{i=1}^{j-1}P_{i}
\end{equation}

And iterative Equation~(\ref{E: iteritiveeq}) becomes:

\begin{equation}\label{E: iteritiveeqt}
	\widetilde{t}_{k+1}=\widetilde{t}_{k}+(m_{k+1})\prod_{i=1}^{k}P_{i}
\end{equation}

The parenthesis around coefficient $m$ represents that $m$ has a range of values.$\\$

As prospective prime numbers in set $S_{k}$ are used to generate prospective prime numbers in $S_{k+1}$, so are prospective twin primes in $S_{k}$ used to generate prospective twin primes in $S_{k+1}$. An individual prospective prime number, not part of a prospective twin prime, in $S_{k}$ cannot be used to generate prospective twin prime, because its other component necessarily has a prime factor $P \le P_{k}$ and cannot be used to generate larger prospective twin primes.

The next two subsections give simple numerical examples of calculating prospective prime numbers and twin primes.

$\\$
\underline{Numerical Example for Numbers Prime to $P \le P_{3}=5$}

Starting with the prospective prime numbers (actual primes in this case) in the first elemental sequence: $N_{0,1}=5$ and $N_{0,3}=7$, which are prime to $2$ and $3$, then in order to calculate other prospective prime numbers, prime to $2$, $3$, and $5$, we have:

\begin{equation}\label{E: proprime2a}
	N_{n,1}=N_{0,1}+\sum_{j=3}^{3}m_{j}\prod_{i=1}^{j-1}P_{i}= 5+m_{3}\cdot6 
\end{equation}
and 
\begin{equation}\label{E: proprime2b}
	N_{n,3}=N_{0,3} +\sum_{j=3}^{3}m_{j}\prod_{i=1}^{j-1}P_{i}= 7+m_{3}\cdot6  
\end{equation}

Then using Equations~(\ref{E: mhat}) find the disallowed values of $m_{j}$:

\[
\widehat{m}_{j}^{(1)}=\frac{\alpha_{1}\cdot 5 - 5\bmod{5}}{6\bmod{5}}=0
\]
\[
\widehat{m}_{j}^{(3)}=\frac{\alpha_{3}\cdot 5 - 7\bmod{5}}{6\bmod{5}}=3
\]

We then get:

\begin{equation}\label{E: proprime3a}
	\left( 
	\begin{array}{c}
		N_{0,1} \\ N_{1,1} \\ N_{2,1} \\ N_{3,1} \\ N_{4,1}
	\end{array}
	\right)
	=
	5+
	\left( 
	\begin{array}{c}
		- \\ 1 \\ 2 \\ 3 \\ 4
	\end{array}
	\right)
	\cdot 6 =
	\left( 
	\begin{array}{c}
		- \\ 11 \\ 17 \\ 23 \\ 29
	\end{array}
	\right)
\end{equation}
and 
\begin{equation}\label{E: proprime3b}
	\left( 
	\begin{array}{c}
		N_{0,3} \\ N_{1,3} \\ N_{2,3} \\ N_{3,3} \\ N_{4,3}
	\end{array}
	\right)
	=
	7+
	\left( 
	\begin{array}{c}
		0 \\ 1 \\ 2 \\-\\ 4
	\end{array}
	\right)
	\cdot 6 =
	\left( 
	\begin{array}{c}
		7 \\ 13 \\ 19 \\ - \\ 31
	\end{array}
	\right)
\end{equation}

Note also that the number pairs, with same value of $n$, are twin primes relative to $P \le 5$ and in this instance are actual twin primes: 

\[
\left(N_{1,1} , N_{1,3} \right)
\]
\[
\left(N_{2,1}, N_{2,3} \right)
\]
\[
\left(N_{4,1},  N_{4,3}\right)
\]
and the following number pairs continue as prospective twin primes, prime to $P \le 5$, for all values of $m_{j} \ge 0$:

\[
\left(N_{1,1}+\sum_{j=4}^{k}m_{j}\prod_{i=1}^{j-1}P_{i} \quad, \quad N_{1,3}+\sum_{j=4}^{k}m_{j}\prod_{i=1}^{j-1}P_{i} \right);
\]
\[
\left(N_{2,1}+\sum_{j=4}^{k}m_{j}\prod_{i=1}^{j-1}P_{i} \quad, \quad N_{2,3}+\sum_{j=4}^{k}m_{j}\prod_{i=1}^{j-1}P_{i} \right);
\]
\[
\left(N_{4,1}+\sum_{j=4}^{k}m_{j}\prod_{i=1}^{j-1}P_{i} \quad, \quad N_{4,3}+\sum_{j=4}^{k}m_{j}\prod_{i=1}^{j-1}P_{i} \right);
\]

$\\$
\underline{Numerical Example for Numbers prime to $P \le P_{4}=7$}

\begin{equation}\label{E: proprime6a}
	(N_{n_{1},1}) =\left(\begin{array}{c}
		N_{1,1} \\ N_{2,1} \\ N_{3,1} \\ N_{4,1}
	\end{array} \right)
	+(m_{4})\prod_{i=1}^{3}P_{i}= \left(\begin{array}{c}
		11 \\17 \\ 23 \\ 29
	\end{array} \right)
	+(m_{4})\cdot 30 
\end{equation}
and
\begin{equation}\label{E: proprime6b}
	(N_{n_{3},3}) = \left( 
	\begin{array}{c}
		N_{0,3} \\ N_{1,3} \\ N_{2,3} \\ N_{4,3}
	\end{array}
	\right)+(m_{4})\prod_{i=1}^{3}P_{i}=\left(\begin{array}{c}
		7 \\ 13 \\ 19 \\ 31
	\end{array}\right)
	+(m_{4})\cdot 30
\end{equation}
Here we use the parenthesis around $m_{4}$ because it is an array of numbers which vary for each $N_{n'_{1},1}$ and $N_{n'_{3},3}$.

Next we find the disallowed values of $m_{4}$ using Equations~(\ref{E: mhat}):

\begin{align}\notag
	\widehat{m}_{4}^{(1)}(11)&=\frac{\alpha_{1}\cdot 7 - 11\bmod{7}}{30\bmod{7}}=\frac{\alpha_{1}\cdot 7 - 4}{2}=5; \; \alpha_{1}=2 \\\notag
	\widehat{m}_{4}^{(1)}(17)&=\frac{\alpha_{1}\cdot 7 - 17\bmod{7}}{30\bmod{7}}=\frac{\alpha_{1}\cdot 7 - 3}{2}=2; \; \alpha_{1}=1 \\\notag
	\widehat{m}_{4}^{(1)}(23)&=\frac{\alpha_{1}\cdot 7 - 23\bmod{7}}{30\bmod{7}}=\frac{\alpha_{1}\cdot 7 - 2}{2}=6; \; \alpha_{1}=2 \\\notag
	\widehat{m}_{4}^{(1)}(29)&=\frac{\alpha_{1}\cdot 7 - 29\bmod{7}}{30\bmod{7}}=\frac{\alpha_{1}\cdot 7 - 1}{2}=3; \; \alpha_{1}=1 \\\notag
	\textrm{and} \\\notag
	\widehat{m}_{4}^{(3)}(7)&=\frac{\alpha_{3}\cdot 7 - 7\bmod{7}}{30\bmod{7}}=\frac{\alpha_{3}\cdot 7 - 0}{2}=0; \; \alpha_{1}=0 \\\notag
	\widehat{m}_{4}^{(3)}(13)&=\frac{\alpha_{3}\cdot 7 - 13\bmod{7}}{30\bmod{7}}=\frac{\alpha_{3}\cdot 7 - 6}{2}=4; \; \alpha_{1}=2 \\\notag
	\widehat{m}_{4}^{(3)}(19)&=\frac{\alpha_{3}\cdot 7 - 19\bmod{7}}{30\bmod{7}}=\frac{\alpha_{3}\cdot 7 - 5}{2}=1; \; \alpha_{1}=1 \\\notag
	\widehat{m}_{4}^{(3)}(31)&=\frac{\alpha_{3}\cdot 7 - 31\bmod{7}}{30\bmod{7}}=\frac{\alpha_{3}\cdot 7 - 3}{2}=2; \; \alpha_{1}=1 \\\notag
\end{align}

The allowed values of $m$ are then:

\begin{equation}\label{E: proprime11a}
	\left(m_{4}^{(1)}\right) = 
	\left(\begin{array}{ccccccc}
		0 & 1 & 2 & 3 & 4 & - & 6 \\
		0 & 1 & - & 3 & 4 & 5 & 6  \\
		0 & 1 & 2 & 3 & 4 & 5 & -  \\
		0 & 1 & 2 & - & 4 & 5 & 6  \\
	\end{array}\right)
\end{equation}

and in the equation for $N_{n,3}$:

\begin{equation}\label{E: proprime11b}
	\left(m_{4}^{(3)} \right)= 
	\left(\begin{array}{ccccccc}
		- & 1 & 2 & 3 & 4 & 5 & 6 \\
		0 & 1 & 2 & 3 & - & 5 & 6  \\
		0 & - & 2 & 3 & 4 & 5 & 6  \\
		0 & 1 & - & 3 & 4 & 5 & 6  \\
	\end{array}\right)
\end{equation}

Using these results in Equations~\ref{E: proprime6a} and \ref{E: proprime6b} gives:

\begin{equation}\label{E: proprime12a}
	(N_{n_{1},1}) =
	\left(\begin{array}{c}
		11 \\17 \\ 23 \\ 29
	\end{array} \right)
	+\left(\begin{array}{ccccccc}
		0 & 1 & 2 & 3 & 4 & - & 6 \\
		0 & 1 & - & 3 & 4 & 5 & 6  \\
		0 & 1 & 2 & 3 & 4 & 5 & -  \\
		0 & 1 & 2 & - & 4 & 5 & 6  \\
	\end{array}\right)
	\cdot 30  
\end{equation}
giving:
\begin{equation}\label{E: proprime13a}
	\left(N_{n_{1},1}\right) =
	\left(\begin{array}{ccccccc}
		11 & 41 & 71 & 101 & 131 & - & 191 \\
		17 & 47 & - & 107 & 137 & 167 & 197 \\
		23 & 53 & 83 & 113 & 143^* & 173 & -  \\
		29 & 59 & 89 & - & 149 & 179 &  209^* \\
	\end{array}\right)
\end{equation}

and

\begin{equation}\label{E: proprime12b}
	\left(N_{n_{3},3}\right) = \left( 
	\begin{array}{c}
		7 \\ 13 \\ 19 \\ 31
	\end{array}
	\right)+\left(\begin{array}{ccccccc}
		- & 1 & 2 & 3 & 4 & 5 & 6  \\
		0 & 1 & 2 & 3 & - & 5 & 6  \\
		0 & - & 2 & 3 & 4 & 5 & 6  \\
		0 & 1 & - & 3 & 4 & 5 & 6  \\
	\end{array}\right)
	\cdot 30
\end{equation}
giving:
\begin{equation}\label{E: proprime13b}
	\left(N_{n_{3},3}\right) =
	\left(\begin{array}{ccccccc}
		- & 37 & 67 & 97 & 127 & 157 & 187^* \\
		13 & 43 & 73 & 103 & - & 163 & 193  \\
		19 & - & 79 & 109 & 139 & 169^* & 199  \\
		31 & 61 & - & 121^* & 151 & 181 &  211 \\
	\end{array}\right)
\end{equation}

Where numbers with an asterix are prime to $p \le 7$, but are not prime numbers, having multiple prime factors $>7$. Note that the two array solutions for $N_{n_{1},1}$ and $N_{n_{3},3}$ contain every prospective prime number $P$ (i.e., prime to all $P \le 7$), where $7 < P < 4+\prod_{i=1}^{4}P_{i}=214 \;$.

These results can be rearranged as follows correspond to the next stage $S_{4}\longrightarrow S_{5}$:

\begin{equation}\label{E: proprime14a}
	\left(N_{n,1}\right)=
	\left(\begin{array}{c}
		N_{1,1} \\N_{2,1}\\N_{3,1} \\N_{4,1}\\N_{6,1}\\N_{7,1}\\N_{8,1}\\N_{9,1}\\N_{11,1}\\N_{13,1}\\N_{14,1}\\N_{16,1}\\N_{17,1}\\N_{18,1}\\N_{21,1}\\N_{22,1}\\N_{23,1}^*\\N_{24,1}\\N_{27,1}\\N_{28,1}\\N_{29,1}\\N_{31,1}\\N_{32,1}\\N_{34,1} 
	\end{array}\right)=
	\left(\begin{array}{c}
		11 \\17\\23 \\29\\41\\47\\53\\59\\71\\83\\89\\101\\107\\113\\131\\137\\143^*\\149\\167\\173\\179\\191\\197\\209^*
	\end{array}\right)
	\quad \textrm{and} \quad
	\left(N_{n,3}\right)=
	\left(\begin{array}{c}
		N_{1,3} \\N_{2,3}\\N_{4,3} \\N_{5,3}\\N_{6,3}\\N_{9,3}\\N_{10,3}\\N_{11,3}\\N_{12,3}\\N_{15,3}\\N_{16,3}\\N_{17,3}\\N_{19,3}\\N_{20,3}\\N_{22,3}\\N_{24,3}\\N_{25,3}^*\\N_{26,3}\\N_{27,3}\\N_{29,3}\\N_{30,3}\\N_{31,3}\\N_{32,3}\\N_{34,3} 
	\end{array}\right)=
	\left(\begin{array}{c}
		13\\19\\31\\37\\43\\61\\67\\73\\79\\97\\103\\109\\121^*\\127\\139\\151\\157\\163\\169^*\\181\\187^*\\193\\199\\211

	\end{array}\right)
\end{equation}\\

These give the following 15 prospective twin primes:

\[
(N_{1,1},N_{1,3}); (N_{2,1},N_{2,3}); (N_{4,1},N_{4,3}); (N_{6,1},N_{6,3}); (N_{9,1},N_{9,3}); (N_{11,1},N_{11,3});  
\]
\[
(N_{16,1},N_{16,3}); (N_{17,1},N_{17,3}); (N_{22,1},N_{22,3}); (N_{24,1},N_{24,3}); (N_{27,1},N_{27,3}); (N_{29,1},N_{29,3}); 
\]
\[
(N_{31,1},N_{31,3}); (N_{321,1},N_{32,3}); (N_{34,1},N_{34,3})
\]

In proceeding to the next stage $S_{4}\longrightarrow S_{5}$, one must keep the prospective prime numbers with asterix, which are not actual prime numbers, since they also generate prospective prime numbers. 

Prospective prime numbers in $S_{5}$ like $121=11^2$ and $143=11 \cdot 13$ which contain prime factor $11$ generate prospective prime numbers in $S_{5}$ prime to $P \le 11$, as long as one excludes $m_{5}=0$, ($\widehat{m}_{5}=0$).

This leads to a general rule: starting with prospective prime numbers in $S_{k}$ to generate prospective prime numbers in $S_{k+1}$. Any prospective prime number in $S_{k}$ that has prime factor $P_{k+1}$ generates a prospective prime number in $S_{k+1}$ using any value of $m_{k+1}$ except $m_{k+1}=0$. Therefore given any $\widetilde{P} \in S_{k}$, where $P_{k+1} | \widetilde{P}$ one can generate a subset of the prospective prime numbers in $S_{k+1}$ using:

\[
\widetilde{P}+\left( \begin{array}{c} 1\\2\\3\\4\\ \cdots \\P_{k+1}-1 \end{array} \right)\cdot \prod_{i=1}^{k}P_{i}
\]

We have shown above that in generating prospective prime numbers one can sort out the prospective twin primes. However, one can generate all prospective twin primes directly using Equation~(\ref{E: proprime1a2}) or Equation~(\ref{E: iteritiveeqt}).

\[
\widetilde{t}_{3}=\left(\begin{array}{c} 5\\7 \end{array}\right)+(m_{j})\prod_{i=1}^{2}P_{i}=
\left(\begin{array}{c} 5\\7  \end{array}\right)+\left(\begin{array}{ccccc}-&1&2&3&4\\0&1&2&-&4  \end{array} \right)\cdot 6
\]
were we used the disallowed values of $m_{j}$ calculated above for the prospective prime number example, giving:

\[
\widetilde{t}_{3}=\left(\begin{array}{ccccc}-&11&17&23&29\\7&13&19&-&31  \end{array} \right) = \left(\begin{array}{ccc}11&17&29\\13&19&31  \end{array} \right)
\]

Then these prospective twin primes (actual twin primes in this example) can be used to generate all the prospective twin primes in $S_{4}$, etc. See Figure~\ref{fig:families-of-pro-twin-primes}.

\begin{figure}
	\centering
	\includegraphics[width=1.0\linewidth]{"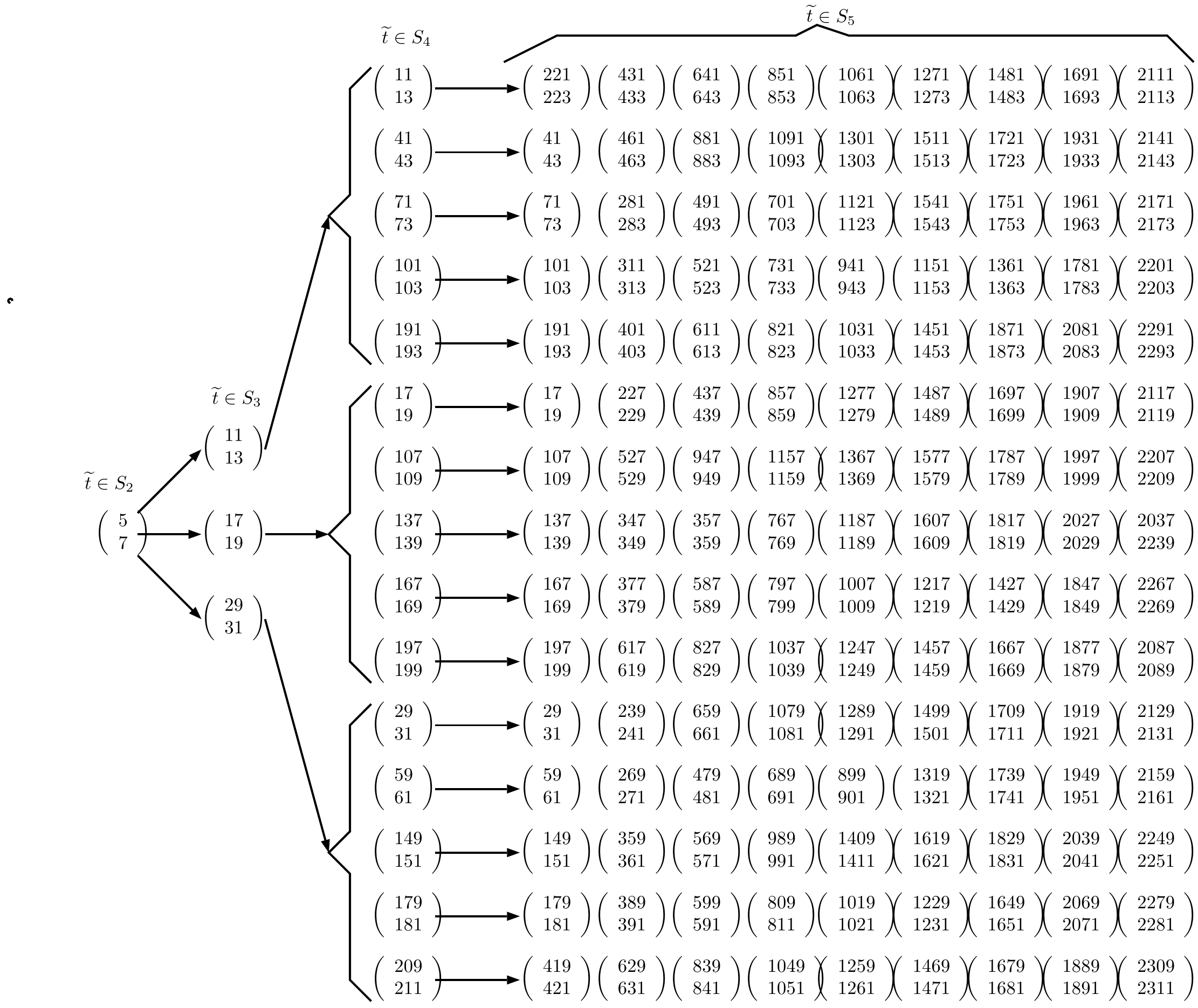"}
	\caption{Branching families of prospective twin primes in the progression of sets: $S_{2}\rightarrow S_{3}\rightarrow S_{4} \rightarrow S_{5}$, where in each set $S_{k}$, prospective twin primes are prime to $P\le P_{k}$.}
	\label{fig:families-of-pro-twin-primes}
\end{figure}

\section{Distribution of Prospective Prime Numbers and Prospective Twin Primes}\label{S: relative_prime_count}

\subsection{Distribution of prospective prime numbers}

We can represent a generic prospective prime number in $S_{k}$ as:

\begin{equation}\label{E: genproprime1}
	\widetilde{P}_{\left\{k\right\}}=\widetilde{P}_{\left\{k-1\right\}} + (m_{k}) \prod_{i=1}^{k-1}P_{i}
\end{equation}
where there are $n_{k-1}^{\widetilde{P}}=\prod_{i=1}^{k-1}(P_{i}-1)$ separate $\widetilde{P}_{\left\{k-1\right\}}$ in $S_{k-1}$ and where $(m_{k})$ represents a spectrum of $P_{k}-1$ numbers out of the $P_{k}$ numbers ranging from $0$ to $P_{k}-1$ thus making Equation~(\ref{E: genproprime1}) an abrieviation for $P_{k}-1$ separate equations for each $\widetilde{P}_{\left\{k-1\right\}}\in S_{k-1}$. 

Looking at $S_{k}$ as a sequence of subsets $S_{k}=\left\{S_{k}^{(0)},S_{k}^{(1)},S_{k}^{(2)},\cdots,S_{k}^{(P_{k}-1)}\right\}$, where $S_{k}^{(j)}=\left\{5+j\cdot \prod_{i=1}^{k-1}P_{i}\longrightarrow 4+(j+1)\cdot \prod_{i=1}^{k-1}P_{i}\right\}$ one can see that for a specific value $(m_{k})=m$:

\[
{\widetilde{P}}_{\{k-1\}} +m \prod_{i=1}^{k-1}P_{i} \in S_{k}^{(m)}
\]

Therefore equation~(\ref{E: genproprime1}) represents that each $\widetilde{P}_{\left\{k-1\right\}}\in S_{k-1}$  generates $P_{k}-1$ prospective prime numbers $\widetilde{P}_{\{k\}}\in S_{k}$, distributed one each to all but one subset $S_{k}^{(m_{k})}$.

Let the one disallowed value of $(m_{k})$ corresponding to $\widetilde{P}_{\{k\}}\bmod{P_{k}}=0$ be $\widehat{m}_{k}$, where:

\begin{equation}\label{E: missingsubset}
	\left[ \widetilde{P}_{\{k-1\}} +\widehat{m}_{k} \prod_{i=1}^{k-1}P_{i}\right]\bmod{P_{k}}=0
\end{equation}

and we can express this as:

\[
\widetilde{P}_{\{k-1\}}\bmod{P_{k}} +\widehat{m}_{k} \left(\prod_{i=1}^{k-1}P_{i}\right)\bmod{P_{k}}=\alpha P_{k}
\]
where $\alpha$ is the smallest integer such that $\widehat{m}_{k}$ is also an integer, giving:  

\begin{equation}\label{E: mhatequ}
	\widehat{m}_{k}=\frac{\alpha P_{k}- \widetilde{P}_{\{k-1\}}\bmod{P_{k}}}{\left(\prod_{i=1}^{k-1}P_{i}\right)\bmod{P_{k}}}
\end{equation}

We can let $\widetilde{P}_{\{k-1\}}\bmod{P_{k}}=\beta$ and let $\left(\prod_{i=1}^{k-1}P_{i}\right)\bmod{P_{k}}=\gamma$, where: $0\le \alpha, \beta \le P_{k}-1$ and $1 \le \gamma \le P_{k}-1$. Giving:

\[
\widehat{m}_{k} = \frac{\alpha P_{k}-\beta}{\gamma}
\]

which requires that: $\alpha P_{k}\equiv \beta\bmod{\gamma}$. This has a unique solution for $\alpha$ dependent on $\beta$ given that $\gamma$ is fixed for $S_{k}$.\cite{FG,HW} Note that $\widehat{m}_{k}=0$ only if $\alpha=0=\beta$, which occurs only when $P_{k} | \widetilde{P}_{\{k-1\}}$.

$\widehat{m}_{k}$ is the disallowed value of $(m_{k})$ in equation~(\ref{E: genproprime1}) which is unique for distinct $\beta=\widetilde{P}_{\{k-1\}}\bmod{P_{k}}$. Any two distinct instances of $\widetilde{P}_{\{k-1\}}$ that differ by a multiple of $P_{k}$ will therefore have the same value of $\widehat{m}_{k}$. Meaning  they have the same disallowed subset and that they each generate one prospective prime number to all the other immediate subsets of $S_{k}$.

Equation~(\ref{E: genproprime1}) gives:

\[
\widetilde{P}_{\left\{k\right\}}\bmod{P_{k}}=\left[\widetilde{P}_{\left\{k-1\right\}}\bmod{P_{k}} + (m_{k}) \left(\prod_{i=1}^{k-1}P_{i}\right)\bmod{P_{k}}\right]\bmod{P_{k}}
\]

The only variable on the right hand side given a choice of $\widetilde{P}_{\{k-1\}}$ is $(m_{k})$, which takes on $P_{k}-1$ distinct values in the range $0\le m_{k} \le P_{k-1}$. Therefore the set of $\widetilde{P}_{\{k\}}$ generated by one instance of $\widetilde{P}_{\{k-1\}}$ represents $P_{k}-1$ distinct residue classes mod $P_{k}$. The only residue class not represented is that corresponding to $(m_{k})=\widehat{m}_{k}$ given in Equation~(\ref{E: mhatequ}). 

In generating the set of all $\widetilde{P}_{\{k\}}$ we start with $n_{k-1}^{\widetilde{P}}=\prod_{i=1}^{k-1}(P_{i}-1)$ prospective prime numbers in $S_{k-1}$. These represent $n_{k-2}^{\widetilde{P}}$ sets of $\widetilde{P}_{\{k-1\}}$ each with $P_{k-1}-1$ members generated by a single $\widetilde{P}_{\{k-2\}}$. Each of these sets, which we represent as $\left\{  \widetilde{P}_{\{k-1\}}\right\}_{\widetilde{P}_{\{k-2\}}}$, represents $P_{k-1}-1$ distinct residue classes mod $P_{k-1}$ and therefore the same number of residue classes mod $P_{k}$.\footnote{
	$N=\left\lfloor \frac{N}{P} \right\rfloor \cdot P + N\bmod{P}=\left\lfloor \frac{N}{P'} \right\rfloor \cdot P' + N\bmod{P'} \longrightarrow N\bmod{P'}=\left\lfloor \frac{N}{P} \right\rfloor \cdot P -\left\lfloor \frac{N}{P'} \right\rfloor \cdot P' + N\bmod{P}$. Therefore, given $P'>P$ there is a one-to-one correspondence of the residue classes mod $P$ to a subset of the residue classes mod $P'$.
}

Equation~(\ref{E: mhatequ}) shows there is a one-to-one correspondence between the residue class of $\widetilde{P}_{\{k-1\}}$ mod $P_{k}$ and $\widehat{m}_{k}$. Each  member of set $\left\{  \widetilde{P}_{\{k-1\}}\right\}_{\widetilde{P}_{\{k-2\}}}$ used to generate $\widetilde{P}_{\{k\}}$ has a corresponding disallowed value $\widehat{m}_{k}$ unique among the members of that set. Therefore, among the $P_{k}$ subsets $S_{k}^{(m_{k})}$ of $S_{k}$, $P_{k-1}-1$ have  $P_{k-1}-2$ prospective prime numbers, $\widetilde{P}_{\{k\}}$, generated by the set $\left\{  \widetilde{P}_{\{k-1\}}\right\}_{\widetilde{P}_{\{k-2\}}}$ and $P_{k}-(P_{k-1}-1)$ have $P_{k-1}-1$  prospective prime numbers generated by that same set. So, every immediate subset of $S_{k}$ has a minimum of $P_{k-1}-2$ prospective prime numbers $\widetilde{P}_{\{k\}}$.

Then given that there are $n_{k-2}^{\widetilde{P}}$ subsets $\left\{ \widetilde{P}_{\{k-1\}}\right\}_{\widetilde{P}_{\{k-2\}}}$, each subset $S_{k}^{(m_{k})}$ of $S_{k}$ has a number of prospective prime numbers $\widetilde{P}_{\{k\}}$ at least equal to:

\begin{align}\label{E: numProPinSub}\notag
	n_{S_{k}^{(m_{k})}}^{\widetilde{P}} &\ge n_{k-2}^{\widetilde{P}} \cdot (P_{k-1}-2)\\\notag
	&= (P_{k-1}-2)\prod_{i=1}^{k-2}(P_{i}-1)\\\notag
	&= [(P_{k-1}-1)-1]\prod_{i=1}^{k-2}(P_{i}-1)\\\notag
	&= \prod_{i=1}^{k-1}(P_{i}-1)-\prod_{i=1}^{k-2}(P_{i}-1)\\
	&= n_{k-1}^{\widetilde{P}}-n_{k-2}^{\widetilde{P}}
\end{align}

Given that result we can state the following theorem:

\begin{theorem}
	Given the set $S_{k}=\left\{N: 5 \le N \le 4+\prod_{i=1}^{k}P_{i}\right\}$ and subsets 
	$S_{k}^{(m)}=\left\{5+m\prod_{i=1}^{k-1}P_{i} \le N \le 4+(m+1)\prod_{i=1}^{k-1}P_{i}\right\}$, $0 \le m \le (P_{k}-1)$, then every subset $S_{k}^{(m)}$ has  $n_{S_{k}^{(m)}}^{\widetilde{P}}\ge n_{k-1}^{\widetilde{P}}-n_{k-2}^{\widetilde{P}}$ prospective prime numbers $\widetilde{P} \in S_{k}^{(m)}$, where $\widetilde{P} = N\in S_{k}$ and $P|N \longrightarrow P > P_{k}$.
\end{theorem}

\begin{proof}
	Let $\widetilde{P}_{\{k\}}$ represent a generic prospective prime number in $S_{k}$. Then consider the subset $\left\{\widetilde{P}_{\{k-1\}}\right\}_{\widetilde{P}_{\{k-2\}}}$ of prospective prime numbers in $S_{k-1}$ generated by a single $\widetilde{P}_{\{k-2\}}$ via the relation:
	
	\[
	\left\{\widetilde{P}_{\{k-1\}}\right\}_{\widetilde{P}_{\{k-2\}}}=\widetilde{P}_{\{k-2\}} + (m_{k-1})\prod_{i=1}^{k-2}P_{i}
	\]
	where the different $\widetilde{P}_{k-1}$ are each distinguished by their values of $m_{k-1}$, where $m_{k-1}$ has $P_{k-1}-1$ valid values among, $0 \le m_{k-1} \le P_{k-1}-1$. Therefore, each $\widetilde{P}_{\{k-1\}}\in \left\{\widetilde{P}_{\{k-1\}}\right\}_{\widetilde{P}_{\{k-2\}}}$ belongs to a unique subset $S_{k-1}^{(m_{k-1})} \subset S_{k-1}$ meaning  $\left\{\widetilde{P}_{\{k-1\}}\right\}_{\widetilde{P}_{\{k-2\}}}$ consists of one prospective prime number in each of $P_{k-1}-1$ of the $P_{k-1}$ subsets of $S_{k-1}^{(m)} \subset S_{k-1}$.
	
	Note that each $\widetilde{P}_{\{k-1\}}\in \left\{\widetilde{P}_{\{k-1\}}\right\}_{\widetilde{P}_{\{k-2\}}}$ belongs to a unique residue class mod $P_{k-1}$ as determined by its associated value of $m_{k-l}$ and therefore also belongs to a unique residue class mod $P_{k}$, given $P_{k}>P_{k-1}$. 
	
	Now consider:
	
	\[
	\left\{\widetilde{P}_{\{k\}}\right\}=\left\{\widetilde{P}_{\{k-1\}}\right\}_{\widetilde{P}_{\{k-2\}}}+(m_{k})\prod_{i=1}^{k-1}P_{i}
	\]
	where we are representing multiple equations, for the $(P_{k-1}-1)$ values of $\widetilde{P}_{\{k-1\}}\in\left\{\widetilde{P}_{\{k-1\}}\right\}_{\widetilde{P}_{\{k-2\}}}$ and the $P_{k}-1$ valid values of $m_{k}$,
	where $m_{k}$ takes the values $0 \le m_{k} \le P_{k}-1$ with one disallowed value given by:
	
	\[
	\widehat{m}_{k}=\frac{\alpha \cdot P_{k}-\left\{\widetilde{P}_{\{k-1\}}\right\}_{\widetilde{P}_{\{k-2\}}}\bmod{P_{k}}}{\left(\prod_{i=1}^{k-1}P_{i}\right)\bmod{P_{k}}}
	\]
	where again we represent multiple equations. $\alpha$ is the least integer in each case that yields an integer for $\widehat{m}_{k}$.
	
	One can see from the equation for $\widehat{m}_{k}$, where each  $\widetilde{P}_{\{k-1\}}\in\left\{\widetilde{P}_{\{k-1\}}\right\}_{\widetilde{P}_{\{k-2\}}}$ represents a unique residue class mod $P_{k}$ and therefore each $\widetilde{P}_{\{k\}}$ derived from that $\widetilde{P}_{\{k-1\}}$ has a unique disallowed subset $S_{k}^{(\widehat{m}_{k})}$. Therefore, for each $\widetilde{P}_{\{k-2\}}$ there are $(P_{k-1}-1)$ individual  $\widetilde{P}_{\{k-1\}}\in \left\{\widetilde{P}_{\{k-1\}}\right\}_{\widetilde{P}_{\{k-2\}}}$ each which generates associated $\widetilde{P}_{k}$ distributed to all subsets of $S_{k}^{(m)} \subset S_{k}$ except for its unique disallowed subset. This leads to a minimum of $(P_{k-1}-2)$ individual $\widetilde{P}_{\{k\}}$ in each subset of $S_{k}^{(m)} \subset S_{k}$ for each $\widetilde{P}_{\{k-2\}}$ in $S_{k-2}$.
	
	Given that there are $n_{k-2}^{\widetilde{P}}=\prod_{i=1}^{k-2}(P_{i}-1)$ individual $\widetilde{P}_{\{k-2\}} \in S_{k-2}$ we get the number total number of $\widetilde{P}_{\{k\}}$ in each subset $S_{k}^{(m)}$ given by:
	
	\begin{equation}
		n_{S_{k}^{(m)}}^{\widetilde{P}} \ge n_{k-2}^{\widetilde{P}} \cdot (P_{k-1}-2) =  n_{k-1}^{\widetilde{P}}-n_{k-2}^{\widetilde{P}}
	\end{equation}
	
\end{proof}

\subsection{Distribution of prospective twin primes}

We know from Theorem~\ref{T: noprotpp} that $S_{k-1}$ contains $n_{k-1}^{\widetilde{t}}=\prod_{i=1}^{k-1}(P_{i}-2)$ prospective twin primes prime to all $P \le P_{k-1}$ and $S_{k}$ contains $n_{k}^{\widetilde{t}}=\prod_{i=1}^{k}(P_{k}-2)$ prospective twin primes prime to all $P \le P_{k}$.  
Let $\widetilde{t}_{k-1}=\left(\begin{array}{c} \widetilde{P}_{n_{k-1},1}\\ \widetilde{P}_{n_{k-1},3} \end{array}\right)$ be a generic prospective twin prime in $S_{k-1}$  and let $\widetilde{t}_{k}=\left(\begin{array}{c} \widetilde{P}_{n_{k},1}\\ \widetilde{P}_{n_{k},3} \end{array}\right)$ be a corresponding prospective twin prime in $S_{k}$ given by:

\[
\widetilde{t}_{k}=\widetilde{t}_{k-1}+(m_{k})\prod_{i=1}^{k-1}P_{i}
\]
where this actually represents separate equations relating $\widetilde{P}_{n_{k},1}$ to $\widetilde{P}_{n_{k-1},1}$ and $\widetilde{P}_{n_{k},3}$ to $\widetilde{P}_{n_{k-1},3}$ both using the same value of $m_{k}$ among the $P_{k}-2$ allowed values of $(m_{k})$, where:

\[
0 \le m_{k} \le P_{k}-1 
\]
and where additionally:

\[
m_{k}\ne \widehat{m}_{k}^{(1)} =\frac{\alpha_{1} P_{k}-\widetilde{P}_{n_{k-1},1}\bmod{P_{k}}}{\left(\prod_{i=1}^{k-1}P_{i}\right)\bmod{P_{k}}} \quad \textrm{and} \quad m_{k}\ne \widehat{m}_{k}^{(3)}= \frac{\alpha_{3} P_{k}-\widetilde{P}_{n_{k-1},3}\bmod{P_{k}}}{\left(\prod_{i=1}^{k-1}P_{i}\right)\bmod{P_{k}}}
\]
where $\widehat{m}_{k}^{(1)}$ and $\widehat{m}_{k}^{(3)}$ specify disallowed sets and $\alpha_{1}$ and $\alpha_{3}$ represent the lowest integer values yielding integer solutions for $\widehat{m}_{k}^{(1)}$ and $\widehat{m}_{k}^{(3)}$.

Each prospective twin prime, prime to all $P\le P_{k-1}$, in $S_{k-1}$ can be used to generate a prospective twin prime, prime to all $P \le P_{k}$, in each subset $S_{k}^{(m_{k})}$ except for the subsets $S_{k}^{\left(\widehat{m}_{k}^{(1)}\right)}$ and  $S_{k}^{\left(\widehat{m}_{k}^{(3)}\right)}$, where the latter are specific to each $\widetilde{t}_{k-1} \in S_{k-1}$. 

Given that $\widetilde{P}_{n_{k-1},3}\bmod{P_{k}}=\left[\widetilde{P}_{n_{k-1},1}\bmod{P_{k}}+2\right]\bmod{P_{k}}$, we always have $\widehat{m}_{k}^{(1)} \ne \widehat{m}_{k}^{(3)}$ starting with any specific $\widetilde{t}_{k-1}$. Therefore, each prospective twin prime of $S_{k-1}$ projects a prospective twin prime into $P_{k}-2$ of the $P_{k}$ subsets $S_{k}^{(m_{k})}$ of $S_{k}$.

Analogous to the approach in the last section we consider
generating all the $\widetilde{t}_{k}\in S_{k}$ from the  $\widetilde{t}_{k-1}\in S_{k-1}$ where the latter consist of $n_{k-2}^{\widetilde{t}}=\prod_{i=2}^{k-2}(P_{i}-2)$ sets of $\left\{\widetilde{t}_{k-1}\right\}_{\widetilde{t}_{k-2}}$.  Where each of the sets $\left\{\widetilde{t}_{k-1}\right\}_{\widetilde{t}_{k-2}}$ consist of $(P_{k-1}-2)$ prospective twin primes in $S_{k-1}$ that are generated from the same $\widetilde{t}_{k-2}$.

The problem in this case is that while the sets of $\widehat{m}_{k-1}^{(1)}$ and $\widehat{m}_{k-1}^{(3)}$ are unique within each set, being derived from one specific $\widetilde{t}_{k-2}$, there can be common values between the two sets for two distinct $\widetilde{t}_{k-1}\in \left\{ \widetilde{t}_{k-1} \right\}_{\widetilde{t}_{k-2}}$.  This can occur when:

\[
\widetilde{t}'_{k-1}=\widetilde{t}_{k-1} \pm (nP_{k}\pm 2)
\]
This condition implies that $P_{n_{k-1},1}\bmod{P_{k}}=P_{n'_{k-1},3}\bmod{P_{k}}$ or $P_{n'_{k-1},1}\bmod{P_{k}}=P_{n_{k-1},3}\bmod{P_{k}}$ and consequently: $\widehat{m}_{k}^{(1)}=\widehat{m}_{k}^{'(3)}$ or $\widehat{m}_{k}^{(3)}=\widehat{m}_{k}^{'(1)}$.

Also, given that the two $\widetilde{t}_{k-1}$ are derived from the same $\widetilde{t}_{k-2}$, we have:

\[
\widetilde{t}_{k-1}=\widetilde{t}_{k-2}+m_{k-1}\prod_{i=1}^{k-2}P_{i}
\]
and
\[
\widetilde{t}'_{k-1}=\widetilde{t}_{k-2}+m'_{k-1}\prod_{i=1}^{k-2}P_{i}
\]
giving:
\[
\Delta \widetilde{t}_{k-1}=
\Delta m_{k-1}\prod_{i=1}^{k-2}P_{i}=
\pm (nP_{k}\pm 2)
\]

Therefore in generating the set of $ \left\{ \widetilde{t}_{k-1} \right\}_{\widetilde{t}_{k-2}}$ instead of having four excluded subsets between a pair of $\widetilde{t}_{\{k-1\}}$ there will be three excluded subsets whenever that pair of prospective twin primes are generated with:

\[
\Delta m_{k-1}=\frac{n^{\pm}P_{k}\pm 2}{\prod_{i=1}^{k-2}P_{i}}
\]
That is $\widetilde{t}_{k-1}$ and $\widetilde{t}'_{k-1}$ may share an excluded subset, $S_{k-1}^{(m)}$. They cannot share the same two excluded subsets.

Letting: $\Delta m_{k-1}^+=\frac{n^{+}P_{k}+ 2}{\prod_{i=1}^{k-2}P_{i}}$ and $\Delta m_{k-1}^-=\frac{n^{-}P_{k}- 2}{\prod_{i=1}^{k-2}P_{i}}$, Gives:

\[
\Delta m_{k-1}^+ + \Delta m_{k-1}^-=\frac{n^{+}P_{k}+ 2}{\prod_{i=1}^{k-2}P_{i}} + \frac{n^{-}P_{k}- 2}{\prod_{i=1}^{k-2}P_{i}}=\frac{(n^{+} + n^{-})P_{k}}{\prod_{i=1}^{k-2}P_{i}}
\]

Since $\Delta m_{k-1}^+ + \Delta m_{k-1}^-$ must be an integer, this requires that $n^{+} + n^{-}=c\cdot \prod_{i=1}^{k-2}P_{i}$, were $c$ is an integer, giving:

\[
n^{+} + n^{-}=c\prod_{i=1}^{k-2}P_{i}
\]
giving:
\[
\Delta m_{k-1}^+ + \Delta m_{k-1}^-=cP_{k}
\] 

Additionally, there are $P_{k-1}$ subsets in $S_{k-1}$ and to be valid $\Delta m_{k-1}^{\pm}<P_{k-1}$, therefore $c=1$. 

As in the last section dealing with the distribution of prospective prime numbers in $S_{k}$ we would like to determine the number of prospective twin primes in the subsets of $S_{k}$. The number of cases where redundancies occur, e.g. $\widehat{m}_{k}^{(1)}=\widehat{m}_{k}^{'(3)}$, which may differ, means we will look for a minimum number.

As shown in the last section for prospective prime numbers, the sets of $\widehat{m}_{k}^{(1)}$ and $\widehat{m}_{k}^{(3)}$ each have internally unique values when generating $\widetilde{t}_{k}$ from the set of $\left\{\widetilde{t}_{k-1}\right\}_{\widetilde{t}_{k-2}}$.  Therefore for the purposes of counting $\widetilde{t}_{k}\in S_{k}^{(m)}$ we can use Table~\ref{T: numbsperset} as a guide for calculating the distribution of prospective twin primes to the subsets of $S_{k}^{(m)} \in S_{k}$.

\begin{table}[h]
	\begin{center}
		\caption{The table shows an allocation of prospective twin primes to subsets where the sets of  $\widehat{m}_{k}^{(1)}$ and $\widehat{m}_{k}^{(3)}$ are each separately unique. It also arbitrarily shows a case where $\Delta m =2$. Such a table for an arbitrary $S_{k}$ would have $P_{k}$ subsets across the top and $P_{k-1}-2$ individual $\widetilde{t}_{k-1}$ down the left-hand column.}
		\renewcommand{\arraystretch}{1.5}
		\begin{tabular}{|c|c|c|c|c|c|c|c|c|c|c|c|}
\hline
&$S_{k}^{(0)}$&$S_{k}^{(1)}$&$S_{k}^{(2)}$&$S_{k}^{(3)}$&$S_{k}^{(4)}$&$S_{k}^{(5)}$&$S_{k}^{(6)}$&$S_{k}^{(7)}$&$\cdots$&$S_{k}^{(P_{k}-2)}$&$S_{k}^{(P_{k}-1)}$ \\
\hline
$\widetilde{t}_{k-1}$&$\widehat{m}_{k}^{(1)}$&$\widetilde{t}_{k}$&$\widehat{m}_{k}^{(3)}$ &$\widetilde{t}_{k}$&$\widetilde{t}_{k}$&$\widetilde{t}_{k}$&$\widetilde{t}_{k}$&$\widetilde{t}_{k}$&$\cdots$&$\widetilde{t}_{k}$&$\widetilde{t}_{k}$   \\
\hline
$\widetilde{t}_{k-1}$&$\widetilde{t}_{k}$&$\widehat{m}_{k}^{(1)}$&$\widetilde{t}_{k}$&$\widehat{m}_{k}^{(3)}$&$\widetilde{t}_{k}$&$\widetilde{t}_{k}$&$\widetilde{t}_{k}$&$\widetilde{t}_{k}$&$\cdots$&$\widetilde{t}_{k}$& $\widetilde{t}_{k}$   \\
\hline
$\widetilde{t}_{k-1}$&$\widetilde{t}_{k}$&$\widetilde{t}_{k}$&$\widehat{m}_{k}^{(1)}$&$\widetilde{t}_{k}$&$\widehat{m}_{k}^{(3)}$&$\widetilde{t}_{k}$&$\widetilde{t}_{k}$&$\widetilde{t}_{k}$&$\cdots$&$\widetilde{t}_{k}$&$\widetilde{t}_{k}$  \\
\hline
$\widetilde{t}_{k-1}$&$\widetilde{t}_{k}$&$\widetilde{t}_{k}$&$\widetilde{t}_{k}$&$\widehat{m}_{k}^{(1)}$&$\widetilde{t}_{k}$&$\widehat{m}_{k}^{(3)}$&$\widetilde{t}_{k}$&$\widetilde{t}_{k}$&$\cdots$&$\widetilde{t}_{k}$&$\widetilde{t}_{k}$  \\
\hline
$\vdots$&$\vdots$&$\vdots$& $\vdots$ &$\vdots$&$\vdots$&$\vdots$  &$\vdots$&$\vdots$&$\cdots$&$\vdots$ &$\vdots$  \\
\hline
$\widetilde{t}_{k-1}$&$\widehat{m}_{k}^{(3)}$&$\widetilde{t}_{k}$&$\widetilde{t}_{k}$& $\widetilde{t}_{k}$ &$\widetilde{t}_{k}$&$\widetilde{t}_{k}$&$\widetilde{t}_{k}$  &$\widetilde{t}_{k}$&$\cdots$&$\widehat{m}_{k}^{(1)}$&$\widetilde{t}_{k}$    \\
\hline
		\end{tabular}
		\label{T: numbsperset} 
	\end{center}
\end{table}

The excluded subsets for a given $\widetilde{t}_{k-1}$ are represented where the $\widehat{m}_{k}^{(1)}$ and $\widehat{m}_{k}^{(3)}$ are shown. With some constraints, the actual distribution to specific subsets doesn't matter for calculating the minimum number of prospective twin primes per subset. Rows may be interchanged and the $\widehat{m}_{j}$ may be shifted along a row, subject to the following considerations: no two $\widehat{m}_{k}^{(1)}$ or two $\widehat{m}_{k}^{(3)}$ can occur in the same column; $\widehat{m}_{j}^{(3)}=\left[\widehat{m}_{j}^{(1)}+\Delta m\right]\bmod{P_{k}}$; and $\widehat{m}_{j}^{(1)}=\widehat{m'}_{j}^{(3)}$ can occur.  

The separation of  $\widehat{m}_{k}^{(1)}$ and $\widehat{m}_{k}^{(3)}$ affects the distribution to subsets and that can be calculated:

\begin{align}\notag
	\widetilde{t}_{k}=&\widetilde{t}_{k-1}+(m_{k})\prod_{i=1}^{k-1}P_{i}\\\notag
	\textrm{expanded gives:}&\\\notag
	\left(\begin{array}{c} P_{n_{k,1}} \\ P_{n_{k},3} \end{array}\right) =& \left(\begin{array}{c} \widetilde{P}_{n_{k-1},1} \\ \widetilde{P}_{n_{k-1},3} \end{array}\right) +  \left(\begin{array}{c} m_{k}^{(1)} \\ m_{k}^{(3)} \end{array}\right)\prod_{i=1}^{k-1}P_{i}
\end{align}

Then relating the two disallowed values of $m_{k}$:

\[
\widehat{m}_{k}^{(1)}=\frac{\alpha_{1}P_{k}-\widetilde{P}_{n_{k-1},1}\bmod{P_{k}}}{\left(\prod_{i=1}^{k-1}P_{i}\right)\bmod{P_{k}}}
\]
and
\begin{align}\notag
	\widehat{m}_{k}^{(3)}=&\frac{\alpha_{3}P_{k}-\widetilde{P}_{n_{k-1},3}\bmod{P_{k}}}{\left(\prod_{i=1}^{k-1}P_{i}\right)\bmod{P_{k}}}\\\notag
	=&\frac{\alpha_{3}P_{k}-(\widetilde{P}_{n_{k-1},1}+2)\bmod{P_{k}}}{\left(\prod_{i=1}^{k-1}P_{i}\right)\bmod{P_{k}}}\\\notag
	=&\frac{\alpha'_{3}P_{k}-\widetilde{P}_{n_{k-1},1}\bmod{P_{k}}-2}{\left(\prod_{i=1}^{k-1}P_{i}\right)\bmod{P_{k}}}\\\notag
	=&\frac{\alpha_{1}P_{k}-\widetilde{P}_{n_{k-1},1}\bmod{P_{k}}}{\left(\prod_{i=1}^{k-1}P_{i}\right)\bmod{P_{k}}}+
	\frac{\Delta\alpha\cdot P_{k}-2}{\left(\prod_{i=1}^{k-1}P_{i}\right)\bmod{P_{k}}}\\\notag
	\widehat{m}_{k}^{(3)}=&\widehat{m}_{k}^{(1)}+\frac{\Delta\alpha\cdot P_{k}-2}{\left(\prod_{i=1}^{k-1}P_{i}\right)\bmod{P_{k}}}
\end{align}

where $\Delta\alpha=\alpha'_{3}-\alpha_{1}$ are integers chosen to give an integer result, and $\alpha'_{3}$ is modified $\alpha_{3}$ to account for removing the $2$ from the mod function.

One can see that the separation of the two disallowed sets is the same for each prospective twin prime in $S_{k}$, therefore justifying the fixed separation shown in Figure~\ref{T: numbsperset}.

For the set $S_{k}$ we have $P_{k}$ subsets and $P_{k-1}-2$ prospective twin primes generated in $S_{k-1}$ for each $\widetilde{t}_{k-2}$. Therefore the left hand column of the corresponding table for $S_{k}$ would have $P_{k-1}-2$ rows of $\widetilde{t}_{k-1}$. Columns (subsets) with no occurrence of $\widehat{m}_{k}^{(1)}$ or $\widehat{m}_{k}^{(3)}$ have $P_{k-1}-2$ prospective twin primes $\widetilde{t}_{k}$. Columns with one occurrence of one or the other of $\widehat{m}_{k}^{(1)}$ or $\widehat{m}_{k}^{(3)}$ have $P_{k-1}-3$ prospective twin primes $\widetilde{t}_{k}$ and columns with  occurrences of both  $\widehat{m}_{k}^{(1)}$ and $\widehat{m}_{k}^{(3)}$ have $P_{k-1}-4$ prospective twin primes $\widetilde{t}_{k}$.

\begin{theorem}\label{T: disrofprotps}
	Given the set of sequential numbers $S_{k}=\left\{N: 5 \le N \le 4+\prod_{i=1}^{k}P_{i}\right\}$ and its $P_{k}$  subsets $S_{k}^{(m)} \subset S_{k}$, where \\$S_{k}^{(m)}=\left\{N: 5+m\prod_{i=1}^{k-1}P_{i} \le N \le 4+(m+1)\prod_{i=1}^{k-1}P_{i}\right\}$, $0 \le m \le P_{k}-1$, and letting $n_{S_{k}^{(m)}}^{\widetilde{t}}$ represent the number of prospective twin primes, $\widetilde{t}_{k}$, in each subset, then: $n_{S_{k}^{(m)}}^{\widetilde{t}}\ge  n_{k-1}^{\widetilde{t}}-2n_{k-2}^{\widetilde{t}}$ and where $P_{k}\le \widetilde{t}_{k} \le 4+\prod_{i=1}^{k}P_{i}$ and $P| \widetilde{t}_{k} \longrightarrow P>P_{k}$.
	
\end{theorem}

\begin{proof}
	Given the preceding discussion we know that for each $\widetilde{t}_{k-2} \in S_{k-2}$ we have a minimum of $P_{k-1}-4$ prospective twin primes in each of the subsets $S_{k}^{(m)}$ which are prime to all $P \le P_{k}$. Then given that there are $n_{k-2}^{\widetilde{t}}=\prod_{i=1}^{k-2}(P_{i}-2)$ prospective twin primes in $S_{k-2}$ and representing the number of prospetive twin primes in $S_{k}^{(m)}$ as $n_{S_{k}^{(m)}}^{\widetilde{t}}$, we have:
	
	\begin{align}\label{E: mintwnprimes}\notag
		n_{S_{k}^{(m)}}^{\widetilde{t}} &\ge n_{k-2}^{\widetilde{t}} \cdot (P_{k-1}-4) \\\notag
		& = (P_{k-1}-4)\prod_{i=1}^{k-2}(P_{i}-2)\\\notag
		&= [(P_{k-1}-2)-2]\prod_{i=1}^{k-2}(P_{i}-2)\\\notag
		&= \prod_{i=1}^{k-1}(P_{i}-2)-2\prod_{i=1}^{k-2}(P_{i}-2)\\
		n_{S_{k}^{(m)}}^{\widetilde{t}}	&\ge  n_{k-1}^{\widetilde{t}}-2n_{k-2}^{\widetilde{t}}
	\end{align}
	
\end{proof}

Given that $S_{k}$ has $P_{k}$ subsets $S_{k}^{(m)}$, the stated minimum number of prospective twin primes in each subset accounts for $P_{k} \cdot(P_{k-1}-4)\prod_{i=3}^{k-2}(P_{i}-2)$ of the $n_{k}^{\widetilde{t}}$ total prospective twin primes in $S_{k}$. Therefore we have:

\[
\frac{P_{k} \cdot n_{S_{k}^{(m)}}^{\widetilde{t}}}{n_{k}^{\widetilde{t}}}
\ge \frac{P_{k}(P_{k-1}-4)\prod_{i=3}^{k-2}(P_{i}-2)}{\prod_{i=3}^{k}(P_{i}-2)}
=
\frac{P_{k}(P_{k-1}-4)}{(P_{k}-2)(P_{k-1}-2)}
\]

Note that $\frac{P_{k}}{P_{k}-2}>1$ and $\frac{P_{k-1}-4}{P_{k-1}-2}<1$ and their product is less than $1$, but is asymptotic to $1$ from below as $P_{k}$ increases.\footnote{
	Let $P_{k}=P_{k-1}+\Delta$, then $\frac{P_{k}(P_{k-1}-4)}{(P_{k}-2)(P_{k-1}-2)}=\frac{1}{1+\frac{2(\Delta+2)}{P_{k}[P_{k}-(\Delta +4)]}}$. Then from the prime number theorem we can approximate the average spacing of prime numbers as $\Delta=\ln{N}$, which shows that the fraction approaches $1$ for large $P_{k}$.
}
This is evident in Table~\ref{T: mintp}.

\begin{table}[h]
	\begin{center}
		\caption{Trend of $\frac{P_{k}(P_{k-1}-4)}{(P_{k}-2)(P_{k-1}-2)}$ as $P_{k}$ increases.}
		\renewcommand{\arraystretch}{1.5}
		\begin{tabular}{|c|c|c|c|c|}
			\hline
			$P_{k}$ & $P_{k-1}$ & $\frac{P_{k}}{P_{k}-2}$ & $\frac{P_{k-1}-4}{P_{k-1}-2}$ & $\frac{P_{k}(P_{k-1}-4)}{(P_{k}-2)(P_{k-1}-2)}$ \\
			\hline
			$11$&$7$  &$\frac{11}{9}$  &$\frac{3}{5}$  &$.73$  \\
			\hline
			$23$  &$19$  &$\frac{23}{21}$  & $\frac{15}{17}$&$.966$  \\
			\hline
			$53$  &$47$  &$\frac{53}{51}$  & $\frac{43}{45}$&$.993$   \\
			\hline
			$103$  &$101$  &$\frac{103}{101}$  & $\frac{97}{99}$&$.991$  \\
			\hline
			$1577$  &$1573$  &$\frac{1577}{1575}$  & $\frac{1569}{1571}$&$.999992$   \\
			\hline
		\end{tabular}
		\label{T: mintp}
	\end{center}
\end{table}

Therefore we can say that prospective twin primes are fairly evenly distributed between the subsets of $S_{k}$ and the minimum distribution of prospective twin primes to each subset $S_{k}^{(m)}$ approaches $\frac{n_{k}^{\widetilde{t}}}{P_{k}}$ as $P_{k}$ increases, where $n_{k}^{\widetilde{t}}$ is the total number of prospective twin primes in $S_{k}$ and $P_{k}$ is the number of subsets, $S_{k}^{(m)}$ in $S_{k}$.

Consider the set of prospective twin primes in $S_{k}$: $\widetilde{T}_{k} \subset S_{k}$, where:\footnote{
	We consider that $P | \widetilde{t}$ if $P$ divides either of the two components of $\widetilde{t}=\left(\begin{array}{c} P_{n,1} \\ P_{n,3}  \end{array}\right)$
}

\[ \widetilde{T}_{k}=\left\{\widetilde{t}: P_{k}< \widetilde{t} \le 4+\prod_{i=1}^{k}P_{i}\quad \& \quad P | \widetilde{t} \rightarrow P>P_{k} \right\} 
\]

If $\widetilde{t} \in \widetilde{T}_{k}$ then $\widetilde{t}=\widetilde{t}_{k}$. We know from the discussion on the distribution of prospective twin primes that the subset $S_{k}^{(0)}$ contains $n_{S_{k}^{(0)}}^{\widetilde{t}}$ prospective twin primes $\widetilde{t}_{k}$. Also since $S_{k}^{(0)}=S_{k-1}$ we know $\widetilde{T}_{k-1}$ contains  $n_{S_{k}^{(0)}}^{\widetilde{t}}$ prospective twin primes $>P_{k}$  that are prime to all $P \le P_{k}$; i.e., there are  $n_{S_{k}^{(0)}}^{\widetilde{t}}$ $\quad \widetilde{t}_{k-1}$ where $\widetilde{t}_{k-1}=\widetilde{t}_{k}$.

The set of prospective twin primes in $S_{k-1}$ is equal to:
\[
\widetilde{T}_{k-1}=\left[\widetilde{T}_{k}\cap \widetilde{T}_{k-1}\right] \cup \left\{\widetilde{t}_{k-1}: P_{k-1} < \widetilde{t}_{k-1} < 4+\prod_{i=1}^{k-1}P_{i} \quad \& \quad P_{k} | \widetilde{t}_{k-1}\right\}
\]

The first set involving the intersection of $\widetilde{T}_{k}$ and $\widetilde{T}_{k-1}$ includes those $\widetilde{t}_{k-1}$ where $\widetilde{t}_{k-1}=\widetilde{t}_{k}$, i.e. where $\widetilde{t}_{k-1}> P_{k}$ and $P_{k} \nmid \widetilde{t}_{k-1}$. The second set includes those $\widetilde{t}_{k-1}$ that contain the prime factor $P_{k}$ including $(P_{k},P_{k+1})$ if it is a twin prime. $\widetilde{t}_{k-1}$ that has $P_{k}$ as a prime factor does generate prospective twin primes in $S_{k}$ via $\widetilde{t}_{k-1}+(m_{k})\prod_{i=1}^{k-1}P_{i}$, but not for $m_{k}=0$ and therefore not in $S_{k}^{(0)}$. We can therefore place a limit on the number of $\widetilde{t}_{k-1}$ that contain prime factor $P_{k}$ by:

\begin{align}\notag
	n_{k-1}^{P_{k}|\widetilde{t}}=&n_{k-1}^{\widetilde{t}}-n_{S_{k}^{(m)}}^{\widetilde{t}}\\\notag
	\le& n_{k-1}^{\widetilde{t}}-[n_{k-1}^{\widetilde{t}}-2n_{k-2}^{\widetilde{t}}]\\\notag
	\le& 2n_{k-2}^{\widetilde{t}}
\end{align}

If we look at the fraction of prospective twin primes in $S_{k-1}$ with prime factor $P_{k}$ we get:

\[
\frac{n_{k-1}^{P_{k}|\widetilde{t}}}{n_{k-1}^{\widetilde{t}}}\le \frac{2n_{k-2}^{\widetilde{t}}}{n_{k-1}^{\widetilde{t}}}=2\frac{\prod_{i=3}^{k-2}(P_{i}-2)}{\prod_{i=3}^{k-1}(P_{i}-2)}=\frac{2}{P_{k-1}-2}
\]

While $n_{k-1}^{P_{k}|\widetilde{t}}$ grows if we pick a larger $P_{k}$, the fraction it represents of all prospective twin primes in $S_{k-1}$ becomes negligible.

This gives an interpretation to the formula in Theorem~\ref{T: disrofprotps}:

\[
n_{S_{k}^{(m)}}^{\widetilde{t}}\ge  n_{k-1}^{\widetilde{t}}-2n_{k-2}^{\widetilde{t}}
\]

The first term on the right is the number of prospective twin primes $\widetilde{t}_{k-1}\in S_{k-1}$ which if there were no disallowed subsets,  contribute equally to all subsets $S_{k}^{(m)}$. The negative term then represents subtracting the number of resulting $\widetilde{t}_{k}=\widetilde{t}_{k-1}+m\prod_{i=1}^{k-1}P_{i}$ in $S_{k}^{(m)}$ where $P_{k}| \widetilde{t}_{k}$. Therefore the negative term accounts for the disallowed subsets which is uniform across all subsets with respects to determining the minimum number of prospective twin primes in each $S_{k}^{(m)}$. Note also, since each subset of $S_{k-1}$ has a uniform minimum of $n_{S_{k-1}^{(m)}}^{\widetilde{t}}$ prospective twin primes, in this respect each subset of $S_{k-1}$ contributes equally to the minimum number of prospective twin primes in each subset of $S_{k}$. This latter point is key to proving Theorem~\ref{T: numtwinprimes}.

Given Theorems \ref{T: noprotpp} and \ref{T: disrofprotps} we can now prove the following theorem.

\begin{theorem}\label{T: numtwinprimes}
	Given integer $l\ge 4$, define $P_{k}=P_{\boldsymbol{\pi}\left(\sqrt{\prod_{i=1}^{l}P_{i}}\right)}$, then let $n_{P_{k}\rightarrow P_{k+1}^2}^{t}$ be the number of twin primes between $P_{k}$ and $P_{k+1}^2$, then:
	\[
	n_{P_{k}\rightarrow P_{k+1}^2}^{t}\ge \left(\prod_{j=l}^{k-1}\frac{P_{j}-4}{P_{j}-2}\right)\prod_{i=3}^{l}(P_{i}-2)=\left(\prod_{j=l}^{k-1}\frac{P_{j}-4}{P_{j}-2}\right)n_{l}^{\widetilde{t}}\\
	\] 
\end{theorem}

\begin{proof}
	Given $l$ and $P_{k}=P_{\boldsymbol{\pi}\left(\sqrt{\prod_{i=1}^{l}P_{i}}\right)}$ consider the set of sequential natural numbers $S_{k}=\left\{5 \longrightarrow 4+\prod_{i=1}^{k}P_{i}\right\}$. We will show that $S_{k}$ always contains prospective twin primes, $\widetilde{t}_{k} \in S_{k}$ prime to all $P \le P_{k}$ where $P_{k}< \widetilde{t}_{k} < P_{k+1}^2$ and consequently $\widetilde{t}_{k}=t_{k}$ is an actual twin prime and the number of such twin primes meets the stated minimum.
	
	Note that while $\prod_{i=1}^{l}P_{i}+1$ is the largest prospective prime number in $S_{l}=\left\{5 \longrightarrow 4+\prod_{i=1}^{l}P_{i}\right\}$ in that it is prime to all $P \le P_{l}$, as a consequence of the multiplication rules $\prod_{i=1}^{l}P_{i}\pm 1$ cannot be the square of a prime number.\footnote{$\prod_{i=1}^{l}P_{i}-1=P_{n,1}$ and $\prod_{i=1}^{l}P_{i}+ 1=P_{n,3}$, where $n=\left[\prod_{i=3}^{l}P_{i}-1\right]$ is even. The multiplication rules (\ref{E: sequenceproducts2}) show that the square of any prime number, where $n_{i}+n_{j}=2n$, has an odd number for its elemental sequence number.} 
	
	With the definition of $P_{k}$ we have: 
	
	\[ 
	P_{k}^2 < \prod_{i=1}^{l}P_{i} \longrightarrow P_{k}^2 \in S_{l}
	\]
	
	and given
	
	$P_{k+1}^2=\left(P_{\boldsymbol{\pi}\left(\sqrt{\prod_{i=1}^{l}P_{i}}\right)+1}\right)^2$, we have:
	
	\[
	\prod_{i=1}^{l}P_{i} < P_{k+1}^2 < \prod_{i=1}^{l+1}P_{i}\longrightarrow  P_{k+1}^2 \in S_{l+1}\quad \& \quad P_{k+1}^2 \notin S_{l}
	\]
	
	Note that $P_{k+1}$ is the smallest prime number whose square is greater than $4+\prod_{i=1}^{l}P_{i}$ and $P_{k}$ is the largest prime number whose square is less than $\prod_{i=1}^{l}P_{i}$.  Therefore all prospective prime numbers and prospective twin primes in $S_{l}$ are less than $P_{k+1}^2$. It remains to show that some $\widetilde{t}_{l}$ are greater than $P_{k}$ and are prime to all $P \le P_{k}$ which means some $\widetilde{t}_{l}=\widetilde{t}_{k}$ or equivalently $\widetilde{T}_{l} \cap \widetilde{T}_{k}\ne \emptyset $, where $\widetilde{T}_{k}$ is the set of all $\widetilde{t}_{k}\in S_{k}$ and $\widetilde{T}_{l}$ is the set of all $\widetilde{t}_{l}\in S_{l}$. In doing this we will show the inequality for $n_{P_{k}\rightarrow P_{k+1}^2}^{t}$ holds.
	
	To prove $\widetilde{T}_{l} \cap \widetilde{T}_{k}\ne \emptyset $ we must trace: $\widetilde{t}_{l} \in S_{l}\rightarrow \widetilde{t}_{l}=\widetilde{t}_{l+1} \in S_{l+1}\rightarrow \cdots \rightarrow \widetilde{t}_{l}=\widetilde{t}_{k} \in S_{k}$. This can only be done through the zeroth subset of each set: $\widetilde{t}_{l} \in S_{l}\rightarrow \widetilde{t}_{l}=\widetilde{t}_{l+1} \in S_{l+1}^{(0)}\rightarrow \cdots \rightarrow \widetilde{t}_{l}=\widetilde{t}_{k} \in S_{k}^{(0)}$, requiring $m_{j}=0$ at each stage of: $\widetilde{t}_{k}=\widetilde{t}_{l}+\sum_{j=l+1}^{k}(m_{j})\prod_{i=j}^{k-1}P_{i}$, otherwise the resulting $\widetilde{t}_{k}$ may be larger than $P_{k+1}^2$. This is straightforward because:
	
	\[
	S_{l}=S_{l+1}^{(0)}\subset S_{l+2}^{(0)}\subset S_{l+3}^{(0)}\cdots \subset S_{k}^{(0)}
	\]	
	
	Based on our definitions, $S_{l+1}^{(0)}$ contains  $n_{S_{l+1}^{(0)}}^{\widetilde{t}}$ prospective twin primes, prime to $P\le P_{l+1}$, where $\widetilde{t}_{l+1}=\widetilde{t}_{l}$. We also know that $S_{l+2}^{(0)}=S_{l+1}$ has  $n_{S_{l+2}^{(0)}}^{\widetilde{t}}$ prospective twin primes, prime to $P\le P_{l+2}$. However all subsets of $S_{l+1}$ have contributed prospective twin primes to $S_{l+2}^{(0)}$. 
	
	 Given that all subsets of $S_{l+1}$ contribute prospective twin primes uniformly to all subsets to $S_{l+2}$ with respects to the minimum number in each subset, the fraction of prospective twin primes in $S_{l+2}^{(0)}$ generated from $\widetilde{t}_{l}=\widetilde{t}_{l+1} \in S_{l+1}^{(0)}$ is given by: 
	
	\[
	\frac{n_{S_{l+1}^{(0)}}^{\widetilde{t}}}{n_{l+1}^{\widetilde{t}}}n_{S_{l+2}^{(0)}}^{\widetilde{t}} \quad \approx \quad \textrm{number of}\quad \widetilde{t}_{l+2}=\widetilde{t}_{l}
	\]
	where we have divided the number of prospective twin primes in $S_{l+1}^{(0)}$ by the total number of prospective twin primes in $S_{l+1}$ as the fraction of prospective twin primes in $S_{l+2}^{(0)}$ generated by prospective twin primes in $S_{l+1}^{(0)}$.
	
	Then we have $n_{S_{l+3}^{(0)}}^{\widetilde{t}}$ prospective twin primes, $\widetilde{t}_{l+3}\in S_{l+3}^{(0)}$ derived from all $\widetilde{t}_{l+2}\in S_{l+2}$. The fraction of those derived from the set of $\widetilde{t}_{l+2}=\widetilde{t}_{l}\in S_{l+2}^{(0)}$ is:
	
	\[
	\frac{n_{S_{l+1}^{(0)}}^{\widetilde{t}}}{n_{l+1}^{\widetilde{t}}}\frac{n_{S_{l+2}^{(0)}}^{\widetilde{t}}}{n_{l+2}^{\widetilde{t}}}n_{S_{l+3}^{(0)}}^{\widetilde{t}} \quad \approx \quad \textrm{number of}\quad \widetilde{t}_{l+3}=\widetilde{t}_{l}
	\]
	
	Carrying this process forward up to the number of $\widetilde{t}_{k}=\widetilde{t}_{l}$, where then $P_{k}<\widetilde{t}_{l}\le P_{k+1}^2$, gives:
	
	\begin{equation}\label{E: numtleqtk}
		n_{p_{k}\rightarrow P_{k+1}^2}^{t}	\approx n_{S_{k}^{(0)}}^{\widetilde{t}} \prod_{j=l+1}^{k-1}\frac{n_{S_{j}^{(0)}}^{\widetilde{t}}}{n_{j}^{\widetilde{t}}}
	\end{equation}
	
	Expanding this using Theorem~\ref{T: disrofprotps} 
	and Theorem~\ref{T: noprotpp} we get:
	
	\begin{align}\label{E: proofdemo}\notag
		n_{p_{k}\rightarrow P_{k+1}^2}^{t} \ge& (P_{k-1}-4)\prod_{i=3}^{k-2}(P_{i}-2) \cdot \prod_{j=l+1}^{k-1}\frac{(P_{j-1}-4)\prod_{i=3}^{j-2}(P_{i}-2)}{\prod_{i=3}^{j}(P_{i}-2)}\\\notag
		=&(P_{k-1}-4)\prod_{i=3}^{k-2}(P_{i}-2) \cdot \left(\prod_{j=l+1}^{k-1}\frac{P_{j-1}-4}{P_{j-1}-2}\right)\left(\prod_{j=l+1}^{k-1}\frac{1}{P_{j}-2}\right)\\\notag
		=& \left(\frac{P_{k-1}-4}{P_{k-1}-2}\right)\left(\prod_{j=l}^{k-2}\frac{P_{j}-4}{P_{j}-2}\right)\prod_{i=3}^{l}(P_{i}-2)\\
		=&\left(\prod_{j=l}^{k-1}\frac{P_{j}-4}{P_{j}-2}\right)\prod_{i=3}^{l}(P_{i}-2) = \left(\prod_{j=l}^{k-1}\frac{P_{j}-4}{P_{j}-2}\right)n_{l}^{\widetilde{t}}
	\end{align}

	The first factor is less than one because its individual factors approach $1$ from below, but there is a disproportionate increase in $k$ when $l$ is increased so there are more terms in the product as $l$ and $k$ increase. The last factor is just the number of prospective twin primes in $S_{l}$, which grows by the factor $(P_{l+1}-2)$ as $l$ increases to $l+1$. Table~\ref{T: etmtwnprime} evaluates inequality (\ref{E: proofdemo}) for a few low values of $l$ showing growing accuracy with $l$ where the second factor appears to grow fast enough to keep a finite and growing product.\footnote{Numbers in the "Actual" column from: https://www.hugin.com.au/prime/twin.php}. We will now show formally that $n_{p_{k'}\rightarrow P_{k'+1}^2}^{t} 
	\ge\; n_{p_{k}\rightarrow P_{k+1}^2}^{t}$.

	\begin{table}[!]
		\begin{center}
			\caption{Calculated minimum number of twin primes  $P_{k} \le t \le P_{k+1}^2$ compared to actual.}
			\renewcommand{\arraystretch}{1.5}
			\begin{tabular}{|c|c|c|c|c|}
				\hline
				$l$	& $k$ &$\substack{n_{p_{k}\rightarrow P_{k+1}^2}^{t}\\ >}$&$\substack{ Actual \: t\\P_{k}\le t_{k} < P_{k+1}^2}$&$\frac{n_{p_{k}\rightarrow P_{k+1}^2}^{t}}{Actual\: t}>$  \\
				\hline
				$4$	& $6$ & $7$&$16$&$.44$ \\
				\hline
				$5$	&$15$  &$43$&$74$&$.58$  \\
				\hline
				$6$	& $40$ & $350$&$480$&$.73$ \\
				\hline
				$7$	&$127$  &$3,988$&$4,653$&$.86$  \\
				\hline
				$8$	& $443$ &$52,432$&$57,529$&$.91$  \\
				\hline
			\end{tabular}
			\label{T: etmtwnprime}
		\end{center}
	\end{table}

	Now consider Equation~(\ref{E: proofdemo}) letting $l\rightarrow l+1$ and $k \rightarrow k'$, where $k=\boldsymbol{\pi}\left(\sqrt{\prod_{i=1}^{l}P_{i}}\right)$ and $k'=\boldsymbol{\pi}\left(\sqrt{\prod_{i=1}^{l+1}P_{i}}\right)$.

	\begin{align}\label{E: proofdemo2}\notag
		n_{p_{k'}\rightarrow P_{k'+1}^2}^{t} 
		\ge&
		\left(\prod_{j=l+1}^{k'-1}\frac{P_{j}-4}{P_{j}-2}\right)n_{l+1}^{\widetilde{t}}\\\notag
		\ge& \left(\prod_{j=l}^{k-1}\frac{P_{j}-4}{P_{j}-2}\right)n_{l}^{\widetilde{t}}\cdot \left(\frac{P_{l}-2}{P_{l}-4}\right)\left(\prod_{j=k}^{k'-1}\frac{P_{j}-4}{P_{j}-2}\right)(P_{l+1}-2)\\\notag
		&\textrm{which gives:}\\
		n_{p_{k'}\rightarrow P_{k'+1}^2}^{t} 
		\ge&\; n_{p_{k}\rightarrow P_{k+1}^2}^{t}\cdot \frac{(P_{l}-2)(P_{l+1}-2)}{(P_{l}-4)}\left(\prod_{j=k}^{k'-1}\frac{P_{j}-4}{P_{j}-2}\right)\\\notag
	\end{align}
	
	Noting that $P_{k}+2\le P_{k+1}$ we can use:
	
	\begin{align}\label{E: interstep1}\notag
		\prod_{j=k}^{k'-1}\frac{P_{j}-4}{P_{j}-2}=\prod_{j=k}^{k'-1}\left(1-\frac{2}{P_{j}-2}\right)
		\ge& \prod_{j=0}^{k'-k}\left(1-\frac{2}{P_{k}+2j-2}\right)\\\notag
		\ge& 1-2\sum_{j=0}^{k'-k}\frac{1}{P_{k}+2(j-1)}\\
		\ge&1-\frac{2}{P_{k}-2}-(k'-k)\frac{2}{P_{k}}
	\end{align}
	
	Then using the definition of $k'$:
	
	\begin{multline}
		k'=\boldsymbol{\pi}\left(\sqrt{\prod_{1}^{l+1}P_{i}}\right)\approx \frac{\sqrt{\prod_{1}^{l+1}P_{i}}}{\ln{\sqrt{\prod_{1}^{l+1}P_{i}}}}=\frac{\sqrt{P_{l+1}}\cdot \sqrt{\prod_{1}^{l}P_{i}}}{\ln{\sqrt{P_{l+1}}}+\ln{\sqrt{\prod_{1}^{l}P_{i}}}}\\
		=\frac{\sqrt{\prod_{1}^{l}P_{i}}}{\ln{\sqrt{\prod_{1}^{l}P_{i}}}}\left(\frac{\sqrt{P_{l+1}}}{1+\frac{\ln{\sqrt{P_{l+1}}}}{\ln{\sqrt{\prod_{1}^{l}P_{i}}}}}\right)
	\end{multline}
	giving:
	\[
	k' \approx k \sqrt{P_{l+1}}
	\]
	
	Using this in result~(\ref{E: interstep1}) gives:
	
	\[
	\prod_{j=k}^{k'-1}\left(1-\frac{2}{P_{j}-2}\right)
	\ge 1-\frac{2}{P_{k}-2}-k\left(\sqrt{P_{l+1}}-1\right)\frac{2}{P_{k}}
	\]
	Then ignoring the $-1$ in the last term and using $k\approx \frac{P_{k}}{\ln{P_{k}}}$ gives:
	
	\[
	\prod_{j=k}^{k'-1}\left(1-\frac{2}{P_{j}-2}\right)
	\ge 1-\frac{2}{P_{k}-2}-\frac{2\sqrt{P_{l+1}}}{\ln{P_{k}}}
	\]
	Using this in the inequality~(\ref{E: proofdemo2}) gives:
	
	\[
	n_{p_{k'}\rightarrow P_{k'+1}^2}^{t}\ge \; n_{p_{k}\rightarrow P_{k+1}^2}^{t}\cdot \frac{(P_{l}-2)(P_{l+1}-2)}{(P_{l}-4)}\left[1-\frac{2}{P_{k}-2}-\frac{2\sqrt{P_{l+1}}}{\ln{P_{k}}}  \right]
	\]
	
	The factor in front of the square brackets is $>(P_{l+1}-2)$.
	Since $P_{k}$ is the largest prime whose square is less than $\prod_{i=1}^{l}P_{i}$, if we take $P_{k}\approx \sqrt{\prod_{i=1}^{l}P_{i}}$ then the last term in the square brackets is seen to be less than $1$ for $l>5$ and gets ever smaller for larger $l$ and the second term in the square brackets is much less than $1$ and gets ever smaller as $l$ and $P_{k}$ increase. Therefore, given Table~\ref{T: etmtwnprime} showing growing values of $n_{p_{k}\rightarrow P_{k+1}^2}^{t}$ one can see that for $l>4$:
	
	\[
	n_{p_{k'}\rightarrow P_{k'+1}^2}^{t}> n_{p_{k}\rightarrow P_{k+1}^2}^{t}
	\]  
	
	This completes the proof.
	
\end{proof}

Given Theorem~\ref{T: numtwinprimes} we can prove the following theorem:

\begin{theorem}
	Given any natural number, $N$ there is always a twin prime greater than $N$.	
\end{theorem}	

\begin{proof}
	Pick integer $l$ so that $P_{k}=P_{\boldsymbol{\pi}\left(\sqrt{\prod_{i=1}^{l}P_{i}}\right)}>N$.Then we know from Theorem~(\ref{T: numtwinprimes}) that there is always a twin prime $>P_{k}$.	
\end{proof}

\end{document}